\documentclass[10pt,letterpaper,twoside]{article}

\usepackage[letterpaper,left=1in,right=1in,top=1.5in,bottom=1.5in]{geometry}

\title{Valuation rings are derived splinters}
\author{Benjamin Antieau and Rankeya Datta}
\date{}

\usepackage[pdfstartview=FitH,
            pdfauthor={Antieau},
            pdftitle={Uniqueness of enhancements},
            colorlinks,
            linkcolor=reference,
            citecolor=citation,
            urlcolor=e-mail,
            backref]{hyperref}

\usepackage{amsmath}
\usepackage{amscd}
\usepackage{amsbsy}
\usepackage{amssymb}
\usepackage{verbatim}
\usepackage{eucal}
\usepackage{microtype}
\usepackage{hyperref}
\usepackage{mathrsfs}
\usepackage{amsthm}
\usepackage{stmaryrd}
\usepackage{xy}
\input xy
\xyoption{all}
\usepackage{tikz}
\usetikzlibrary{matrix,arrows}
\usepackage{dsfont}
\usepackage[T1]{fontenc}
\usepackage{mathbbol}

\usepackage{fancyhdr}
\usepackage{color}
\definecolor{todo}{rgb}{1,0,0}
\definecolor{conditional}{rgb}{0,1,0}
\definecolor{e-mail}{rgb}{0,.40,.80}
\definecolor{reference}{rgb}{.20,.60,.22}
\definecolor{mrnumber}{rgb}{.80,.40,0}
\definecolor{citation}{rgb}{0,.40,.80}

\renewcommand{\rm}{\mdseries}

\let\oldmarginpar\marginpar
\renewcommand\marginpar[1]{\-\oldmarginpar[\raggedleft\footnotesize #1]%
{\raggedright\footnotesize #1}}

\newcommand{\Dscr}{\mathcal{D}}

\newcommand{\Fscr}{\mathcal{F}}

\newcommand{\Oscr}{\mathcal{O}}

\newcommand{\Tscr}{\mathcal{T}}

\newcommand{\FF}{\mathds{F}}

\newcommand{\QQ}{\mathds{Q}}
\newcommand{\RR}{\mathds{R}}

\newcommand{\ZZ}{\mathds{Z}}

\newcommand{\qc}{\mathrm{qc}}
\newcommand{\coh}{\mathrm{coh}}

\DeclareMathOperator*{\colim}{colim}
\newcommand{\iso}{\cong}

\newcommand{\perf}{\mathrm{perf}}
\newcommand{\Hom}{\mathrm{Hom}}

\usepackage{mathtools}

\DeclareMathOperator{\Spec}{Spec}

\DeclareMathOperator{\Frac}{Frac}
\DeclareMathOperator{\Ext}{Ext}

\DeclareMathOperator{\coker}{coker}
\DeclareMathOperator{\id}{id}
\DeclareMathOperator{\QCoh}{QCoh}
\DeclareMathOperator{\Mod}{Mod}

\theoremstyle{plain}
\newtheorem{theorem}[subsubsection]{Theorem}
\newtheorem{proposition}[subsubsection]{Proposition}
\newtheorem{lemma}[subsubsection]{Lemma}

\newtheorem{corollary}[subsubsection]{Corollary}

\theoremstyle{definition}
\newtheorem{definition}[subsubsection]{Definition}

\newtheorem{remark}[subsubsection]{Remark}

\theoremstyle{definition}

\theoremstyle{plain}
\newtheorem{innercustomthm}{Theorem}
\newenvironment{customthm}[1]
  {\renewcommand\theinnercustomthm{#1}\innercustomthm}
  {\endinnercustomthm}

\theoremstyle{plain}  
\newtheorem{innercustomcor}{Corollary}
\newenvironment{customcor}[1]
  {\renewcommand\theinnercustomcor{#1}\innercustomcor}
  {\endinnercustomcor}
  
\theoremstyle{plain}  
\newtheorem{innercustomprop}{Proposition}
\newenvironment{customprop}[1]
  {\renewcommand\theinnercustomprop{#1}\innercustomprop}
  {\endinnercustomprop}
  
\numberwithin{equation}{subsubsection}
\numberwithin{equation}{subsubsection}

\usepackage{tikz-cd}

\newcommand{\cF}{\mathcal{F}}

\newcommand{\cI}{\mathcal{I}}

\newcommand{\bL}{\mathbf{L}}

\newcommand{\cO}{\mathcal{O}}

\newcommand{\bR}{\mathbf{R}}

\newcommand{\bV}{\mathbf{V}}

\newcommand{\fp}{\mathfrak{p}}

\makeatletter
\def\@tocline#1#2#3#4#5#6#7{\relax
  \ifnum #1>\c@tocdepth 
  \else
    \par \addpenalty\@secpenalty\addvspace{#2}%
    \begingroup \hyphenpenalty\@M
    \@ifempty{#4}{%
      \@tempdima\csname r@tocindent\number#1\endcsname\relax
    }{%
      \@tempdima#4\relax
    }%
    \parindent\z@ \leftskip#3\relax \advance\leftskip\@tempdima\relax
    \rightskip\@pnumwidth plus4em \parfillskip-\@pnumwidth
    #5\leavevmode\hskip-\@tempdima
      \ifcase #1
       \or\or \hskip 1em \or \hskip 2em \else \hskip 3em \fi%
      #6\nobreak\relax
    \dotfill\hbox to\@pnumwidth{\@tocpagenum{#7}}\par
    \nobreak
    \endgroup
  \fi}
\makeatother

\begin{document}

\maketitle

\begin{abstract}
    \noindent
    We give three proofs that valuation rings are derived splinters: a
    geometric proof using absolute integral closure, a homological proof which
    reduces the problem to checking that valuation rings are splinters (which
    is done in the second author's PhD thesis and which we reprise here), and a
    proof by approximation which reduces the problem to Bhatt's proof of the
    derived direct summand conjecture. The approximation property also shows
    that smooth algebras over valuation rings are splinters.
\end{abstract}

\section{Introduction}

Recall that a subring $V$ of a field $K$ is called a \emph{valuation ring of $K$}
if for all $x \in K$, either $x$ or $x^{-1}$ is an element of $V$. Valuation
rings, while typically non-Noetherian, satisfy many desirable properties of
Noetherian regular rings. For example, the (absolute) Frobenius endomorphism
of a valuation ring of prime characteristic is flat \cite{DS16}, while a 
well-known result of Kunz says that a Noetherian ring of prime characteristic
has flat Frobenius precisely when the ring is regular \cite{Kun69}. More recently, Kelly 
and Morrow have shown that the algebraic $K$-theory of valuation rings exhibits
behavior similar to the algebraic $K$-theory of regular rings \cite{KM18}. The main 
goal of this
paper is to show that valuation rings satisfy another recently established property of regular
rings, namely that of being a derived splinter (see Definition \ref{def:d-splinter}). 
More specifically, we show the following.

\begin{customthm}{\ref{thm:valuations-derived-splinters}}
Let $V$ be a valuation ring of arbitrary characteristic. Then for any
proper surjective morphism $f: X \rightarrow \Spec(V)$, the induced
map $\cO_{\Spec(V)} \rightarrow \bR f_*\cO_X$ admits a section
in $\Dscr_\qc(\Spec(V))$.
\end{customthm}

Bhatt first defined derived splinters in his thesis \cite{Bha10} (see also \cite{Bha12})
and showed later using techniques from perfectoid geometry that regular rings
are derived splinters \cite{Bha18}, thereby proving a derived analogue of 
Hochster's famous direct summand conjecture (now a theorem due to Hochster
 \cite{Hoc73} and Andr\'e~\cite{And18}). 
However, the notion of a derived splinter
was studied much earlier by Kov\'acs who 
showed that a variety over a field of characteristic $0$ is a 
derived splinter precisely when it has rational singularities \cite{Kov00}. 
Derived splinters are also related to the well-studied commutative algebra
notion of splinters (see Definition \ref{def:splinters}). For instance, 
the notions of splinters and derived
splinters coincide for prime characteristic Noetherian schemes \cite{Bha12}, 
and forthcoming work of Bhatt suggests the 
same is true for excellent schemes in mixed characteristic. In fact,
Theorem \ref{thm:valuations-derived-splinters} 
suggested itself when
the second author gave a characteristic independent
proof of the splinter property for valuation rings \cite{Dat17(a)},
reproduced here as Proposition \ref{prop:valuations-splinters}.

We provide
three proofs of Theorem \ref{thm:valuations-derived-splinters}: a geometric proof 
using the absolute integral closure, a homological proof which
reduces the problem to valuation rings being splinters, and a
proof by regular approximation which reduces the problem to Bhatt's proof of the
derived direct summand conjecture. All three proofs have a unifying
feature in that they exploit striking homological
properties of valuation rings.
The principal among these is that finitely generated torsion free modules over a 
valuation ring
are free (see Lemma \ref{lem:fin-gen-torsion free}). 
This simple property has the following pleasing global consequence, 
reminiscent of the behavior of proper maps over regular local rings.

\begin{customcor}{\ref{cor:valuation-rings-universally-cohesive}}
Given a valuation ring $V$ and a finitely presented proper morphism 
$f: X \rightarrow \Spec(V)$,  if $F \in \Dscr^b_\coh(X)$, then
    $\bR f_*F$ is a perfect complex on $\Spec(V)$.
\end{customcor}

\noindent Using the observation that finitely presented modules over a valuation ring
have projective dimension at most $1$ leads to a decomposition of
$\bR f_*F$ as a direct sum of its cohomology modules (Corollary \ref{cor:fp-modules-pd-1}),
allowing us to reduce one proof of Theorem \ref{thm:valuations-derived-splinters}
to the splinter property.

Finiteness properties of the right derived pushforward are also at the heart
of our proof that the derived splinter condition satisfies faithfully flat 
descent under some mild restrictions.

\begin{customprop}{\ref{prop:descent-derived-splinter}} 
If $A \rightarrow B$ is a faithfully flat map such that 
$B$ is a derived splinter, then $A$ is a derived splinter
provided it is universally cohesive.
\end{customprop}

\noindent Universal cohesiveness is a stable coherence property 
(Definition \ref{def:universally-cohesive}) that
ensures, by work of Fujiwara and Kato, that higher direct images of
finitely presented proper morphisms behave as
in the Noetherian world 
(Theorem \ref{thm:Fujiwara-Kato-finiteness}). 
A result of Nagata 
(Theorem \ref{thm:Nagata-amazing}), generalized substantially in 
\cite{raynaud-gruson}, implies that valuation rings 
are universally cohesive. Using faithfully flat descent
one then reduces the other proofs 
of the derived splinter property for valuation rings to the case where 
the ring is absolutely integrally closed, a setting 
where a wealth of geometric techniques become available.
For example, the valuative criterion 
can be used to show that absolutely integrally closed valuation rings satisfy 
the splitting condition for a more general class of morphisms than 
just proper surjective ones (Lemma \ref{lem:lifting-identity}). 

The most interesting aspect of absolutely integrally closed valuation rings
in this paper is their relation with local uniformization. Classical local uniformization
due to Zariski \cite{Zar40} implies that any valuation ring of characteristic $0$
can be expressed as a directed colimit of smooth $\QQ$-subalgebras, that is,
such a ring is \emph{ind-regular}. Using Temkin's
inseparable local uniformization \cite{temkin-inseparable}
and his substantial refinement of de Jong's alterations \cite{temkin-tame}, 
we show that there are important classes of valuation rings in prime 
and mixed characteristic that are ind-regular 
(Propositions \ref{prop:indsmooth} and \ref{prop:indregular}). For example, an
unpublished observation of Elmanto and Hoyois implies that
\emph{all} absolutely integrally closed valuation rings are ind-regular
(Corollary \ref{cor:absoluteintegralclosed}) and are even directed colimits of their
regular Noetherian subalgebras. This property is not shared by
other `nice' absolutely integrally closed domains (Remark \ref{rem:not-indsmooth}).
Since regular rings are derived direct summands \cite{Bha18}, 
descent properties of morphisms over directed limits of
schemes then gives another proof of the fact that absolutely integrally closed 
valuation rings are derived splinters.

Examining the behavior of the derived splinter property under localization 
(Lemma \ref{lem:d-splinter-local})
also reveals that Pr\"ufer domains, the global analogues of
valuation rings, are derived splinters (Corollary \ref{cor:Prufer}). 
In particular, the absolute 
integral closure of a valuation ring is a derived splinter, 
a fact, that as far as we are aware, is not known for absolute integral
closures of regular rings. Furthermore, upon a careful consideration 
of what the correct
notion of splinter should be in a non-Noetherian context (Lemma \ref{lem:splinter-new-def}
and Definition \ref{def:splinters}), we show
using the regular approximation property that smooth algebras 
over valuation rings are splinters (Theorem \ref{thm:smooth-valuations}). 
However,
we do not know whether smooth algebras over valuation rings are derived
splinters.

\subsection*{Acknowledgements} We thank Bhargav Bhatt, Linquan Ma, Akhil
Mathew, Matthew Morrow, Takumi Murayama, 
Emanuel Reinecke and Kevin Tucker for helpful conversations. 
The first author is indebted to Elden Elmanto who pointed out~\cite{temkin-tame}
and suggested that Proposition~4.2.1 should be true.
The second author is especially grateful to Emanuel Reinecke 
for numerous illuminating discussions about derived splinters and for making us
aware of Fujiwara and Kato's finiteness result (Theorem \ref{thm:Fujiwara-Kato-finiteness}).
Additionally, we thank Takumi for comments on a draft.
The first author was supported by NSF Grant DMS-1552766.

\section{Preliminaries}

Let $X$ be a scheme
We will use the derived category
$\Dscr_\qc(X)\subseteq\Dscr(\mathrm{Mod}_{\Oscr_X})$ consisting of the full
subcategory on those chain complexes of
$\Oscr_X$-modules with quasicoherent cohomology sheaves.
If $\bullet=+,-,b$, then $\Dscr^\bullet_\qc(X)$ is the full triangulated subcategory of $\Dscr_\qc(X)$ of
complexes with bounded below cohomology sheaves, bounded above cohomology sheaves, or
bounded cohomology sheaves, respectively. 

For a ring $R$, we say a map of $R$-modules $M \rightarrow N$ is \emph{pure} 
if for any $R$-module $P$, the induced map $M \otimes_R P \rightarrow 
N \otimes_R P$ is injective. Pure maps are also known as \emph{universally
injective} maps in the literature; see \cite[\href{https://stacks.math.columbia.edu/tag/058I}{Tag 058I}]{stacks-project}. For example, any faithfully flat ring map $R \rightarrow S$ is pure
 and a ring map $R \rightarrow S$ that splits in $\Mod_R$
is also pure. A result of Mesablishvili \cite{Mes00} shows that a 
ring map $\varphi: R \rightarrow S$
is pure if and only if $\varphi$ is an effective descent morphism for modules
(see also \cite[\href{https://stacks.math.columbia.edu/tag/08WE}{Tag 08WE}]{stacks-project}
for some history pertaining to Mesablishvili's result).

\subsection{Derived pushforward and flat pullback}
As is customary, we will use the abbreviation qcqs for a morphism of schemes $f: X \rightarrow Y$ that is quasi-compact and quasi-separated.

\begin{theorem}
\label{thm:derived-pushforward-bounded-cpx}
Let $f: X \rightarrow Y$ be a qcqs morphism of schemes, and suppose $Y$ is quasi-compact. Then for any $F \in \Dscr^-_{\qc}(X)$ (resp. $F \in \Dscr^b_{\qc}(X)$), $\bR f_*F \in \Dscr^-_{\qc}(Y)$ (resp. $\bR f_*F \in \Dscr^b_{\qc}(X)$). In particular, $\bR f_*\mathcal{O}_X \in \Dscr^b_{\qc}(Y)$.
\end{theorem}

\begin{proof}
For a textbook reference of this result, see \cite[Prop. 3.9.2]{Lip09}. The main point is that if $Y$ is a quasi-compact scheme, then there exists an integer $d$ such that for any quasi-coherent sheaf $\cF$ of $\cO_X$-modules, $\text{R}^if_*\cF = 0$ for $i > d$. 
\end{proof}

We next recall the comparison isomorphism associated with Tor-independent cartesian squares.

\begin{theorem}
\label{thm:derived-pushforward-fpqc-pullback}
Consider a cartesian square
    $$\xymatrix{
        X'\ar[r]^{g'}\ar[d]^{f'}&X\ar[d]^f\\
        Y'\ar[r]^g&Y
    }$$ of quasi-separated schemes where all morphisms are qcqs.
    \begin{enumerate}
        \item[{\em (a)}] There is an induced natural transformation $$\bL g^*\bR
            f_*\rightarrow\bR f'_*\bL{g^{'}}^*$$ of functors
            $\Dscr_{\qc}(X)\rightarrow\Dscr_{\qc}(Y')$.
        \item[{\em (b)}] If $Y'$ and $X$ are Tor-independent (for example, if $f$ or $g$ is 
            flat), then the above natural transformation is a natural isomorphism. In particular, 
            $\bL g^* \bR f_*\cO_X \simeq \bR(f')_*\cO_{X'}$.
    \end{enumerate}
\end{theorem}

\newcommand{\Torscr}{\Tscr\mathrm{or}}

\begin{proof}
For a proof see \cite[Thm. 3.10.3]{Lip09}. Here Tor-independence means that for all $x \in X$ and $y' \in Y$ such that $f(x) =  g(y') \coloneqq y$,
one has
$$\Torscr^{\cO_{Y,y}}_i(\cO_{Y',y'}, \cO_{X,x}) = 0,$$
for all $i > 0$. Thus, Tor-independence is automatic if $f$ or $g$ is flat. The
    second assertion of (b) follows from the Tor-independent isomorphism because
    $\bL {g'}^*\cO_X \simeq {g'}^*\cO_X = \cO_{X'}$.
\end{proof}

\subsection{Coherent sheaves and cohesive schemes} Recall that given a scheme
$X$, an $\cO_X$-module $\cF$ is \emph{coherent} if
\begin{enumerate}
    \item[(a)] $\cF$ is of finite type and
    \item[(b)] for any open subset $U \subseteq X$ and any morphism of $\cO_U$-modules $\eta: \cO^{\oplus n}_{U} \rightarrow \cF|_U$, the kernel $\ker(\eta)$ is a finite type $\cO_U$-module.
\end{enumerate}

Note that coherent $\cO_X$-modules are quasi-coherent. Schemes $X$ for which
the structure sheaf $\cO_X$ is coherent as a module over of itself will be of
special interest in this paper, and will be given a special name following
Fujiwara and Kato \cite[Chap. 0, Def. 5.1.1]{FK18}.

\begin{definition}[\cite{FK18}]\label{def:cohesive-scheme}
A scheme $X$ is \emph{cohesive} if $\cO_X$ is coherent as a module over itself.
\end{definition}

\begin{remark}
An affine scheme $X = \Spec(A)$ is a cohesive scheme if and only if $A$ is a coherent ring, that is, if and only if every finitely generated ideal of $A$ is finitely presented as an $A$-module. Since finitely generated ideals of a valuation ring are principal (hence free of rank $1$), it follows that the spectrum of a valuation ring is a cohesive affine scheme. Spectra of valuation rings will be the principal example of cohesive schemes for us in this paper.
\end{remark}

We will need to understand how coherence behaves under operations such as pullbacks and pushforwards. While coherence is not preserved under pullbacks of arbitrary morphisms of schemes, the story is different when the schemes are cohesive because of the following alternate characterization of coherence property.

\begin{lemma}
\label{lem:coherent-modules-cohesive-schemes}
If $X$ is a cohesive scheme, then an $\cO_X$-module $\cF$ is coherent if and only if $\cF$ is finitely presented as an $\cO_X$-module.
\end{lemma}

\begin{proof}
While it is clear that coherence implies finite presentation, the converse follows from the coherence of $\cO_X$ and the fact that the category of coherent $\cO_X$-modules is an abelian sub-category of the category of quasi-coherent $\cO_X$-modules, and so is closed under cokernels.
\end{proof}

Since finite-presentation behaves well under pullbacks, one immediately obtains the following.

\begin{corollary}
\label{cor:pullback-cohesive-scheme}
Let $f: X \rightarrow Y$ be a morphism of cohesive schemes. If $\cF$ is a coherent $\cO_Y$-module then $f^*\cF$ is a coherent $\mathcal{O}_X$-module. 
\end{corollary}

On the other hand, coherence always descends under faithfully flat and quasi-compact pullback, regardless of whether the schemes are cohesive.

\begin{proposition}
\label{prop:fpqc-descent}
Let $f: X \rightarrow Y$ be a faithfully flat and quasi-compact morphism of schemes.
\begin{enumerate}
    \item[{\em (i)}] If $\cF$ be a quasi-coherent sheaf of $\cO_Y$-modules and
        $f^*\cF$ is coherent, then $\cF$ is coherent.
    \item[{\em (ii)}] If $X$ and $Y$ are cohesive and $F \in \Dscr_{\qc}(Y)$, then $F \in \Dscr^\bullet_{\coh}(Y) \Leftrightarrow f^*F \in \Dscr^\bullet_{\coh}(X)$ for $\bullet = +, -, b$. 
\end{enumerate}
\end{proposition}

\begin{proof}
    The property of being of finite type satisfies fpqc descent \cite[\href{https://stacks.math.columbia.edu/tag/05AZ}{Tag 05AZ}]{stacks-project}. Thus, $\cF$ is of finite type because $f^*\cF$ if of finite type. Now suppose $U \subseteq Y$ is an open subset and $\eta: \cO_U^{\oplus} \rightarrow \cF|_U$ is a map of $\cO_U$-modules. Since $f|_{f^{-1}(U)}: f^{-1}(U) \rightarrow U$ is also flat (and quasi-compact)  
    $$\ker(f|_{f^{-1}(U)}^*\eta) = f|_{f^{-1}(U)}^*\ker(\eta)$$
    is of finite type because $f^*\cF$ is coherent. Then by descent applied to the fpqc morphism $f|_{f^{-1}(U)}$, we have that $\ker(\eta)$ is of finite type, proving coherence of $\cF$. 
    
    For (ii), since $f$ is faithfully flat, $H^i(f^*F) \simeq f^*H^i(F) = 0$ if and only if
    $H^i(F) = 0$ for any $F \in \Dscr_{\qc}(Y)$ (here we use quasi-coherence of
    the cohomology sheaves). Therefore $f^*F \in \Dscr^\bullet_{\qc}(X)$ if and
    only if $F \in \Dscr^\bullet_{\qc}(Y)$, for $\bullet = +, -, b$.
    Furthermore, since $X, Y$ are both cohesive, Corollary
    \ref{cor:pullback-cohesive-scheme} and part (i) of this proposition imply
    that $H^i(f^*F) = f^*H^i(F)$ is coherent if and only if $H^i(F)$ is
    coherent. This completes the proof of (ii).
\end{proof}

\subsection{Universal cohesiveness} Universal cohesiveness is a stable coherence property of schemes.
While the notion is not new, the terminology was introduced recently by
 Fujiwara--Kato \cite[Chap. 0, Def.
5.1.1]{FK18}.

\begin{definition}\label{def:universally-cohesive}
A scheme $X$ is \emph{universally cohesive} if
for any morphism $f: Y \rightarrow X$ that is locally of finite presentation,
the scheme $Y$ is cohesive, that is, $\cO_Y$ is a coherent sheaf. We will say
that a ring $R$ is \emph{universally cohesive} if $\Spec(R)$ is universally
cohesive.
\end{definition}

For example, any locally Noetherian scheme is universally cohesive.
Furthermore, if $X$ is universally cohesive, then any scheme $Y$ that is
locally of finite presentation over $X$ is also universally cohesive. In particular,
the quotient of a universally cohesive ring $A$ by a finitely generated is
universally cohesive.  Note, 
a universally cohesive scheme is cohesive. Our interest in
universally cohesive schemes stems from the fact that valuation rings are
universally cohesive, a well-known fact that we will soon prove as a
consequence of a finiteness result of Nagata (Theorem
\ref{thm:Nagata-amazing}).

The following straightforward lemma shows that universal cohesiveness is
preserved under arbitrary localizations.

\begin{lemma}
\label{lem:polynomial-extensions}
A ring $A$ is universally cohesive if and only if for all $n > 0$, 
the polynomial ring $A[x_1,\dots,x_n]$ is coherent. In particular,
if $A$ is universally cohesive, then $S^{-1}A$ is universally cohesive for
any multiplicative set $S \subset A$.
\end{lemma}

\begin{proof}
For the if and only if statement, the non-trivial assertion is the 
backward implication. Since a finitely presented $A$-algebra
is a quotient of a polynomial $A$-algebra $A[x_1,\dots,x_n]$ by a 
finitely generated ideal for some $n > 0$, and since
quotients of coherent rings by finitely generated ideals are coherent
\cite[Thm. 2.4.1(1)]{Gla89}, the backward implication follows.

A polynomial ring
over $S^{-1}A$ is a localization of a polynomial ring over $A$.
As coherence is preserved under localization \cite[Thm. 2.4.2]{Gla89}, 
it follows that universal cohesiveness does as well.    
\end{proof}

Higher direct images of proper, finitely presented morphisms over universally
cohesive schemes satisfy many of the pleasing finiteness properties of higher
direct images of proper morphisms over Noetherian schemes, as shown by the
following result of Fujiwara and Kato~\cite[Chap. I, Thm. 8.1.3]{FK18}.
Below, for a coherent scheme $Y$, we let $\Dscr^b_\coh(Y)$ denote the full triangulated subcategory of
$\Dscr^b_\qc(Y)$ consisting of complexes with bounded coherent cohomology sheaves.

\begin{theorem}[\cite{FK18}]
\label{thm:Fujiwara-Kato-finiteness}
Let $f: X \rightarrow Y$ be a proper, finitely presented morphism, where $Y$ is
    a universally cohesive quasi-compact scheme (thus $X$ is also universally
    cohesive and quasi-compact). If $\cF\in\Dscr^b_\coh(X)$, then
$$\bR f_*\cF\in \Dscr^b_{\coh}(Y).$$
In particular, if $\cF$ is a finitely presented $\cO_X$-module, then $\bR f_*\cF \in \Dscr^b_{\coh}(Y)$, that is, the higher direct images of $\cF$ are coherent (equivalently, finitely presented).
\end{theorem}

\begin{remark}
Theorem \ref{thm:Fujiwara-Kato-finiteness} is well-known if $f: X \rightarrow
    Y$, in addition to satisfying the hypotheses of loc. cit., is also flat.
    Indeed, when $X$ (resp. $Y$) is cohesive, the objects of
    $\Dscr^b_{\coh}(X)$ (resp. $\Dscr^b_{\coh}(Y)$) are precisely the
    \emph{pseudo-coherent} objects of $\Dscr_{\qc}(X)$ (resp. $\Dscr_{\qc}(Y)$)
    with bounded cohomology~\cite[p. 115, Cor. 3.5 b)]{SGA6}. By
    \cite[\href{https://stacks.math.columbia.edu/tag/0CSD}{Tag
    0CSD}]{stacks-project}, $\bR f_*$ preserves pseudo-coherence because $f$ is
    flat, proper and finitely presented. Moreover, $\bR f_*$ maps
    $\Dscr^b_{\qc}(X)$ into $\Dscr^b_{\qc}(Y)$ because $Y$ is quasi-compact
    (Theorem \ref{thm:derived-pushforward-bounded-cpx}), and so, one concludes
    that $\bR f_*(\Dscr^b_{\coh}(X)) \subseteq \Dscr^b_{\coh}(Y)$. In our
    investigation of the derived splinter condition for valuation rings, we
    will often be able to assume that $f$ is flat and finitely presented in
    addition to being proper because of Corollary
    \ref{cor:dominating-proper-by-flat}. Hence we will not need the full
    strength of Theorem \ref{thm:Fujiwara-Kato-finiteness}.
\end{remark}

\subsection{Descent and localization of the derived splinter property}
The notion of a derived splinter, introduced by Bhatt \cite{Bha10}, is a strengthening of 
the notion of a splinter. Since Bhatt was primarily studying Noetherian schemes when he 
introduced derived splinters, we feel a minor modification of his definition is more 
natural in a  non-Noetherian setting.

\begin{definition}  
\label{def:d-splinter}
A scheme $S$ is a \emph{derived splinter} if for every proper, finitely presented and 
surjective morphism $f: X \rightarrow S$, the induced map $\cO_S \rightarrow \bR f_*\cO_X$ 
splits in $\Dscr_\qc(S)$. We will say a ring $A$ is a \emph{derived splinter} if $\Spec(A)$ has 
this property.
\end{definition}

\begin{remark}
Our definition of a derived splinter differs from Bhatt's definition 
\cite[Def. 1.3]{Bha12} in that 
we only require $\cO_S \rightarrow \bR f_*\cO_X$ to split for proper, surjective 
morphisms that are finitely presented, instead of all proper, surjective morphisms. 
\end{remark}

The goal of this subsection is to demonstrate that the derived splinter property satisfies faithfully flat descent with appropriate coherence assumptions. The main result is the following.

\begin{proposition}
\label{prop:descent-derived-splinter}
Let $\varphi: \Spec(B) \rightarrow \Spec(A)$ be a faithfully flat map such that $\Spec(A)$ is universally
cohesive. If $B$ is a derived splinter then so is $A$.
\end{proposition}

The proof of Proposition \ref{prop:descent-derived-splinter} will exploit the following basic algebraic fact.

\begin{lemma}
\label{lem:injective-maps-cohomology}
Let $f: A \rightarrow B$ be a faithfully flat ring homomorphism. If $C^\bullet$ is a complex of $A$-modules, then the induced maps on cohomology $H^i(C^\bullet) \rightarrow H^i(C^\bullet \otimes_A B)$ are injective for all $i$.
\end{lemma}

\begin{proof}
Since cohomology commutes with flat base change, we have $H^i(C^\bullet \otimes_A B) \cong H^i(C^\bullet) \otimes_A B$, and the induced map $H^i(C^\bullet) \rightarrow H^i(C^\bullet \otimes_A B)$ can be identified under this isomorphism with the map $H^i(C^\bullet) \rightarrow H^i(C^\bullet) \otimes_A B$ obtained by tensoring $A \rightarrow B$ by $H^i(C^\bullet)$. Since $A \rightarrow B$ is faithfully flat, $H^i(C^\bullet) \rightarrow H^i(C^\bullet) \otimes_A B$ is injective by \cite[Chap. I, $\mathsection 3.5$, Prop. 9]{Bou89}, completing the proof.
\end{proof}

\begin{remark}
\label{rem:injective-cohomology-pure-pullback}
A more general version of Lemma \ref{lem:injective-maps-cohomology} appears in \cite[Lem. $2.1^\circ$]{HR76} where it is shown that if $K^i \coloneqq \coker(C^{i-1} \xrightarrow{d^{i-1}} C^i)$ and the maps $K^i \rightarrow K^i \otimes_A B$ are injective for all $i$ (for example, if $A \rightarrow B$ is pure), then the induced maps $H^i(C^\bullet) \rightarrow H^i(C^\bullet \otimes_A B)$ are injective.
\end{remark} 

Let $A$ be a commutative ring. Recall that an object $C^\bullet$ of $\Dscr(A)$ is \emph{pseudo-coherent} if it is isomorphic in $\Dscr(A)$ to a bounded above complex $E^\bullet$ of finite free $A$-modules. Since $E^\bullet$ is $K$-projective, note that any isomorphism $E^\bullet \rightarrow C^\bullet$ in $\Dscr(A)$ can be represented by an honest quasi-isomorphism of complexes \cite[\href{https://stacks.math.columbia.edu/tag/064B}{Tag 064B}]{stacks-project}.

\begin{remark}
\label{rem:pseudo-coherent-over-coherent-rings}
An $A$-module $M$ considered as the complex $M[0]$ concentrated in degree $0$ is pseudo-coherent precisely if $M$ has a resolution (possibly infinite) by finite free $A$-modules \cite[\href{https://stacks.math.columbia.edu/tag/064T}{Tag 064T}]{stacks-project}. In particular, if $M[0]$ is pseudo-coherent, then $M$ is a finitely presented $A$-module, and the converse holds if $A$ is coherent. Moreover, if $A$ is coherent, then a complex $C^\bullet$ of $A$-modules is pseudo-coherent if and only if the cohomology of $C^\bullet$ is bounded above and $H^i(C^\bullet)$ is a finitely presented (equivalently, coherent) $A$-module for all $i$ \cite[\href{https://stacks.math.columbia.edu/tag/0EWZ}{Tag 0EWZ}]{stacks-project}.
\end{remark}

Our main interest in pseudo-coherent complexes of $A$-modules is because of the next result.

\begin{proposition}
\label{prop:injective-maps-derived-cat}
Let $f: A \rightarrow B$ be a faithfully flat ring homomorphism. Let $C^\bullet, K^\bullet \in \Dscr(A)$ be complexes such that $C^\bullet$ is pseudo-coherent and $K^\bullet$ has bounded below cohomology. Then the induced map
\[
\Hom_{\Dscr(A)}(C^\bullet, K^\bullet) \rightarrow \Hom_{\Dscr(B)}(C^\bullet \otimes_A B, K^\bullet \otimes_A B)
\]
is injective.
\end{proposition}

\begin{proof}
By \cite[\href{https://stacks.math.columbia.edu/tag/0A6A}{Tag 0A6A}(3)]{stacks-project} and flatness of $f$, we have an isomorphism in $\Dscr(B)$
\[
\bR\Hom_A(C^\bullet, K^\bullet) \otimes_A B \xrightarrow{\sim} \bR\Hom_B(C^\bullet \otimes_A B, K^\bullet \otimes_A B).
\]
Note that at the heart of the above isomorphism is the basic commutative algebra fact that if $M$ is a finitely-presented $A$-module, then for any $A$-module $N$, $\Hom_A(M,N) \otimes_A B \simeq \Hom_B(M \otimes_A B, N \otimes_A B)$ if $B$ is $A$-flat.

By Lemma \ref{lem:injective-maps-cohomology}, the map of complexes 
\[
\bR\Hom_A(C^\bullet, K^\bullet) \rightarrow \bR\Hom_A(C^\bullet, K^\bullet) \otimes_A B
\]
induces injective maps on cohomology. In particular, we then have an injection
\[
H^0\bR\Hom_A(C^\bullet, K^\bullet) \hookrightarrow H^0(\bR\Hom_A(C^\bullet, K^\bullet) \otimes_A B) \simeq H^0\bR\Hom_B(C^\bullet \otimes_A B, K^\bullet \otimes_A B).
\]
But $H^0\bR\Hom_A(C^\bullet, K^\bullet) = \Hom_{\Dscr(A)}(C^\bullet, K^\bullet)$ and $H^0\bR\Hom_B(C^\bullet \otimes_A B, K^\bullet \otimes_A B) = \Hom_{\Dscr(B)}(C^\bullet \otimes_A B, K^\bullet \otimes_A B)$ \cite[\href{https://stacks.math.columbia.edu/tag/0A64}{Tag 0A64}]{stacks-project}. This completes the proof.
\end{proof}

We can now prove Proposition \ref{prop:descent-derived-splinter}.

\begin{proof}[Proof of Proposition \ref{prop:descent-derived-splinter}]
Let $f: X \rightarrow \Spec(A)$ be a proper, surjective, finitely presented morphism. We want to show that the canonical map 
\[
\cO_{\Spec(A)} \rightarrow \bR f_*\cO_X
\]
splits in $\Dscr_{\qc}(\Spec(A))$. Completing $\cO_{\Spec(A)} \rightarrow \bR f_*\cO_X$ to an exact triangle
\begin{equation}
\label{eq:exact-triangle-descent}
\cO_{\Spec(A)} \rightarrow \bR f_*\cO_X \rightarrow C^\bullet \xrightarrow{w} \cO_{\Spec(A)}[1],
\end{equation}
note that $\cO_{\Spec(A)} \rightarrow \bR f_*\cO_X$ splits if and only if $w = 0$ \cite[Cor. 1.2.7, Rem. 1.2.9]{Nee01}.
Now consider the pullback square
\begin{center}
\begin{tikzcd}
  X_B \arrow[r, "\phi"] \arrow[d, "f_B"]
    & X \arrow[d, "f"] \\
  \Spec(B) \arrow[r, "\varphi"]
&\Spec(A). \end{tikzcd}
\end{center}
Applying $\bL\varphi^* = \varphi^*$ to the exact triangle in (\ref{eq:exact-triangle-descent}) and using Theorem \ref{thm:derived-pushforward-fpqc-pullback}, we get an exact triangle
\[
\cO_{\Spec(B)} \rightarrow \bR(f_B)_*\cO_{X_B} \rightarrow \varphi^*C^\bullet \xrightarrow{\varphi^*w} \cO_{\Spec(B)}[1].
\]
Since $\Spec(B)$ is a derived splinter and the base change map $f_B: X_B \rightarrow \Spec(B)$ is proper, surjective and finitely presented, it follows that $\cO_{\Spec(B)} \rightarrow \bR(f_B)_*\cO_{X_B}$ splits, or equivalently, $\varphi^*w = 0$. We want to show that this implies $w = 0$, that is, $\cO_{\Spec(A)} \rightarrow \bR f_*\cO_X$ splits.

The complex $\bR f_*\cO_X$ has coherent cohomology by Theorem \ref{thm:Fujiwara-Kato-finiteness} and universal cohesiveness of $\Spec(A)$. Consequently, $C^\bullet$ also has coherent cohomology because coherent modules form a weak Serre subcategory of the category of $\cO_{\Spec(A)}$-modules and $\cO_{\Spec(A)}$ is itself coherent by assumption \cite[\href{https://stacks.math.columbia.edu/tag/01BY}{Tag 01BY}]{stacks-project}. In particular, $C^\bullet$ is an object of $\Dscr_{\qc}(\Spec(A))$. Identifying $C^\bullet$ with the complex of $\cO_{\Spec(A)}$-modules associated to the complex of $A$-modules 
$\bR\Gamma(\Spec(A),C^\bullet)$
via the equivalence between $\Dscr_{\qc}(\Spec(A))$ and $\Dscr(\QCoh(\Spec(A))$
        \cite[\href{https://stacks.math.columbia.edu/tag/08D9}{Tag
        08D9}(3)]{stacks-project}, it follows that $\bR\Gamma(\Spec(A),C^\bullet)$ is a
        pseudo-coherent object of $\Dscr(A)$
        \cite[\href{https://stacks.math.columbia.edu/tag/0EWZ}{Tag 0EWZ}]{stacks-project}.
        Finally, 
        using the commutativity of the diagram 
\begin{center}
\begin{tikzcd}
  \Dscr(A) \arrow[r, "\sim"] \arrow[d, "- \otimes_A B"]
    &  \Dscr_{\qc}(\cO_{\Spec(A)}) \arrow[d, "\varphi^*"] \\
  \Dscr(B) \arrow[r, "\sim"]
&\Dscr_{\qc}(\cO_{\Spec(B)}) 
\end{tikzcd}
\end{center}
as shown in \cite[\href{https://stacks.math.columbia.edu/tag/08DW}{Tag 08DW}]{stacks-project}, we see that the induced pullback map
$$\Hom_{\Dscr(\Spec(A))}(C^\bullet, \cO_{\Spec(A)}[1]) \longrightarrow \Hom_{\Dscr(\Spec(B))}(\varphi^*C^\bullet, \cO_{\Spec(B)}[1])$$
is injective by Proposition \ref{prop:injective-maps-derived-cat}. Therefore, $\varphi^*w = 0$ implies $w = 0$, as desired.
\end{proof}

\begin{remark}
Even if $\Spec(A)$ is not universally cohesive, the proof of Proposition \ref{prop:descent-derived-splinter} shows when $f: X \rightarrow \Spec(A)$ is a proper, surjective, finitely presented morphism 
such that $\bR f_*\cO_X$ is pseudo-coherent, then $\cO_{\Spec(A)} \rightarrow \bR f_*\cO_X$ splits 
as long as $B$ is a faithfully flat extension of $A$ that is a derived splinter. Indeed, the cone $C^\bullet$ of $\cO_{\Spec(A)} \rightarrow \bR f_*\cO_X$ is then pseudo-coherent by 
\cite[\href{https://stacks.math.columbia.edu/tag/08CD}{Tag 08CD}]{stacks-project} because 
$\cO_{\Spec(A)}$ is trivially pseudo-coherent, and the rest of the proof applies verbatim.
\end{remark}

A similar flat base change argument shows that for a universally cohesive ring $A$, the derived
splinter property behaves well with respect to localization.

\begin{lemma}
\label{lem:d-splinter-local}
Let $A$ be a universally cohesive domain. The following are equivalent:
\begin{enumerate}
	\item[{\em (1)}] $A$ is a derived splinter.
	\item[{\em (2)}] $A_\fp$ is a derived splinter for each prime ideal
$\fp$ of $A$.
	\item[{\em (3)}] For any multiplicative set $S \subset A$, $S^{-1}A$ is 
	a derived splinter.
\end{enumerate}
\end{lemma}

\begin{proof}
The implication $(3) \Rightarrow (2)$ is clear. 
We next show $(2) \Rightarrow (1)$, for which 
we do not need $A$ to be a domain. Let $f: X \rightarrow \Spec(A)$ be a proper, surjective,
finitely presented morphism. 
Again, using the equivalence  $\Dscr(A) \xrightarrow{\sim} \Dscr_{\qc}(\Spec(A))$ given by 
\cite[\href{https://stacks.math.columbia.edu/tag/06Z0}{Tag 06Z0}]{stacks-project}, 
it suffices to show that the induced map 
$$f^\sharp: A \rightarrow \bR\Gamma(X,\cO_X)$$ 
splits in $\Dscr(A)$. Note that splitting of $f^\sharp$ is equivalent to surjectivity 
of the map of $A$-modules
\begin{equation}
\label{eq:surjection}
\Hom_{\Dscr(A)}(\bR\Gamma(X,\cO_X),A) \xrightarrow{\_ \circ f^\sharp} \Hom_{\Dscr(A)}(A,A).
\end{equation}
 
By Fujiwara--Kato's finiteness result (Theorem \ref{thm:Fujiwara-Kato-finiteness}), 
we know that $\bR\Gamma(X,\cO_X)$
is a pseudo-coherent object of $\Dscr(A)$ since it has coherent cohomology
\cite[\href{https://stacks.math.columbia.edu/tag/0EWZ}{Tag 0EWZ}]{stacks-project}. Therefore,
by the fact that cohomology commutes with flat base change, we get
\begin{align*}
\Hom_{\Dscr(A)}(\bR\Gamma(X,\cO_X), A) \otimes_A A_\fp & \simeq 
H^0\big{(}\bR\Hom_A(\bR\Gamma(X,\cO_X), A)\otimes_A A_\fp\big{)} \\
& \simeq H^0\big{(}\bR\Hom_{A_\fp}(\bR\Gamma(X,\cO_X)\otimes_A A_\fp, A_\fp)\big{)}\\
& \simeq \Hom_{\Dscr(A_\fp)}(\bR\Gamma(X,\cO_X)\otimes_A A_\fp, A_\fp),
\end{align*}
where the second isomorphism follows by 
\cite[\href{https://stacks.math.columbia.edu/tag/0A6A}{Tag 0A6A}(3)]{stacks-project} using
pseudo-coherence of $\bR\Gamma(X,\cO_X)$. 
Now by Theorem \ref{thm:derived-pushforward-fpqc-pullback}(b) and again by flat base
change, we have 
$$\bR\Gamma(X,\cO_X)\otimes_{A} A_\fp \simeq \bR\Gamma(X_{A_\fp}, \cO_{X_{A_\fp}}),$$
in $\Dscr(A_\fp)$, where $X_{A_\fp} \coloneqq X \times_{\Spec(A)} \Spec(A_\fp)$. 
The upshot is that
$$\Hom_{\Dscr(A)}(\bR\Gamma(X,\cO_X), A) \otimes_A A_\fp \simeq 
\Hom_{\Dscr(A_\fp)}(\bR\Gamma(X_{A_\fp},\cO_{X_{A_\fp}}), A_\fp)$$
as $A_\fp$-modules, and similarly, $\Hom_{\Dscr(A)}(A,A) \otimes_A A_\fp \simeq 
\Hom_{\Dscr(A_\fp)}(A_\fp,A_\fp)$. Since surjectivity of the $A$-linear map in
(\ref{eq:surjection}) can be  checked after localizing at each prime ideal $\fp$ \cite[Prop. 3.9]{AM}, 
the desired result follows because $A_\fp \rightarrow \bR\Gamma(X_{A_\fp},\cO_{X_{A_\fp}})$ 
splits for each $\fp$.

It remains to prove $(1) \Rightarrow (3)$. For this implication we do not need $A$
to be universally cohesive. Let $\varphi: X \rightarrow \Spec(S^{-1}A)$ be
a proper, surjective, finitely presented morphism. Note that $S^{-1}A$ is the 
filtered colimit of the system of principal localizations $\{A_f: f \in S\}$
with flat transition maps 
\cite[\href{https://stacks.math.columbia.edu/tag/00CR}{Tag 00CR}]{stacks-project}.
By \cite[\href{https://stacks.math.columbia.edu/tag/01ZM}{Tag 01ZM}]{stacks-project},
there exists $f_0 \in S$ and morphism of finite presentation 
$$\varphi_0: X_0 \rightarrow \Spec(A_{f_0})$$
such that $X = X_0 \times_{A_{f_0}} S^{-1}A$. After replacing $f_0$ 
by a further multiple, we may assume by 
\cite[\href{https://stacks.math.columbia.edu/tag/07RR}{Tags 07RR}, 
\href{https://stacks.math.columbia.edu/tag/081F}{081F}]{stacks-project}
that $\varphi_0: X_{f_0} \rightarrow \Spec(A_{f_0})$ is also proper and surjective.
If $\cO_{\Spec(A_{f_0})} \rightarrow \bR \varphi_{0,*} \cO_{X_0}$ splits,
then by flat pullback 
$\cO_{\Spec(S^{-1}A)} \rightarrow \bR\varphi_*\cO_X$ splits as well. Thus, 
we may assume $S^{-1}A$ is a principal localization $A_f$. 

It suffices to  
show that there exists a proper, finitely presented morphism 
$\varphi': X' \rightarrow \Spec(A)$
and a pullback square
\begin{equation}
\label{diag:pullback-sq}
\xymatrix{
        X\ar[r]^{i'}\ar[d]^{\varphi}& X'\ar[d]^{\varphi'}\\
        \Spec(A_f)\ar[r]^{i}& \Spec(A).
    }
\end{equation}
Indeed, the generic point of $\Spec(A)$
must then be in the closed image of $\varphi'$, which implies $\varphi'$ is surjective. 
Since $A$ is a derived splinter by the hypothesis of (1),
$\cO_{\Spec(A)} \rightarrow \bR(\varphi')_*\cO_{X'}$ splits,
inducing a splitting of $\cO_{\Spec(A_f)} \rightarrow \bR\varphi_*\cO_X$
by flat pullback along $i$.
To show the existence of (\ref{diag:pullback-sq}), 
a short argument using Noetherian approximation on $A$ and by descent of
properness, finite presentation and surjectivity along filtered limits of 
schemes, we may further
assume that $A$ is Noetherian. Nagata compactification 
applied to the separated, finite type morphism 
$X \xrightarrow{\varphi} \Spec(A_f) \xrightarrow{i} \Spec(A)$
\cite[\href{https://stacks.math.columbia.edu/tag/0F41}{Tag 0451}]{stacks-project} 
gives a commutative diagram
\begin{equation}\label{diag:Nag-compact}
\xymatrix{
        X\ar[r]^{i'}\ar[d]^{\varphi}& \overline{X}\ar[d]^{\varphi'}\\
        \Spec(A_f)\ar[r]^{i}& \Spec(A)
    }
\end{equation}
where $i'$ is an open immersion and $\varphi': \overline{X} \rightarrow \Spec(A)$
is a proper morphism.
Furthermore, replacing $\overline{X}$ by the scheme theoretic image of $X$ in $\overline{X}$,
we may assume that $i'(X)$ is dense in $\overline{X}$
\cite[\href{https://stacks.math.columbia.edu/tag/01RG}{Tag 01RG}]{stacks-project}. 
This implies that 
(\ref{diag:Nag-compact}) is already a pullback square. 
To see this, let $W \coloneqq (\varphi')^{-1}(\Spec(A_f))$
and consider the open immersion $i': X \rightarrow W$ (by shrinking $\overline{X}$). Then 
$X \xrightarrow{i'} W \rightarrow \Spec(A_f)$ is proper because
it equals $\varphi$
and $W \rightarrow \Spec(A_f)$ is proper by base change. 
Thus, $X \xrightarrow{i'} W$ is proper 
\cite[\href{https://stacks.math.columbia.edu/tag/01W6}{Tag 01W6}]{stacks-project}.
In particular, $i'(X)$ is closed and dense in 
$W$, 
and so, $i'(X) = W$.
But this precisely means that $\varphi$ is the pullback of $\varphi'$ by $i$.
\end{proof}

\section{Some finiteness properties of valuation rings}

\subsection{Coherent modules over valuation rings}
Although valuation rings are typically highly non-Noetherian, they exhibit behavior that often characterizes the class of regular local rings in the Noetherian world. For example, as a consequence of a result of Osofsky \cite{Oso67} it follows that any valuation ring of finite Krull dimension has global dimension $\leq 2$ and the global dimension is $\leq 1$ precisely when the valuation ring is Noetherian (see also \cite{Dat17} for some applications).

Our principal source of finiteness results for modules over valuation rings is the following well-known fact which roughly says that valuation rings often behave like principal ideal domains.

\begin{lemma}
\label{lem:fin-gen-torsion free}
Any finitely generated torsion free module over a valuation ring is free. Consequently, a module over a valuation ring is flat if and only if it is torsion free.
\end{lemma}

\begin{proof}
That finitely generated torsion free modules over a valuation ring are free can
    be proved via a Nakayama lemma type argument and can be found in
    \cite[Chap. VI, $\mathsection 3.6$, Lem. 1]{Bou89}. Flat modules over
    valuation rings are torsion free because valuation rings are domains.
    Conversely, any torsion free module over a valuation ring is a filtered
    colimit of its finitely generated submodules, which are all free by the
    first part of this lemma. Now, any torsion free module is flat since a filtered colimit of flat modules is flat.
\end{proof}

As a corollary we obtain an upper bound on the projective dimension of finitely
presented modules over a valuation ring, which in turn implies that bounded
complexes with coherent cohomology decompose as a direct sum of their
cohomology modules.

\begin{corollary}
\label{cor:fp-modules-pd-1}
Let $M$ be a finitely presented (equivalently, coherent) module over a valuation ring $V$. Then the projective dimension of $M$ is $\leq 1$. In particular, 
\begin{enumerate}
    \item[{\rm (1)}] every object of $\Dscr^b_{\coh}(\Spec(V))$ is a perfect
        complex and
    \item[{\rm (2)}] if $\cF \in \Dscr^b_{\coh}(\Spec(V))$, then $F \simeq \bigoplus_{i} H^i(\cF)[-i]$.
\end{enumerate}
\end{corollary}

\begin{proof}
    Choose a surjective $R$-linear map $\varphi: R^{\oplus n} \twoheadrightarrow M$
    for some $n > 0$. Then $\ker(\varphi)$ is a finitely generated
    $R$-submodule of $R^{\oplus n}$ because $M$ is finitely presented. However,
    $\ker(\varphi)$ is also torsion free since it is a submodule of a
    torsion free module. Therefore $\ker(\varphi)$ is a free $R$-module by
    Lemma \ref{lem:fin-gen-torsion free}, and consequently, $M$ has a free
    resolution of length $\leq 1$, proving the desired bound on its projective
    dimension. In particular, $M$ is a perfect $R$-module
    \cite[\href{https://stacks.math.columbia.edu/tag/066Q}{Tag
    066Q}]{stacks-project}.

    Now, if $F \in \Dscr^b_{\coh}(\Spec(V))$, then all its cohomology modules
    $H^i(F)$ are finitely presented (Lemma
    \ref{lem:coherent-modules-cohesive-schemes}), so they are sheaves associated to
    perfect modules by what we proved in the previous paragraph. Then $F$ is a
    perfect complex since it has a finite filtration with perfect graded
    pieces; see
    \cite[\href{https://stacks.math.columbia.edu/tag/08EB}{Tag
    08EB}]{stacks-project} and
    \cite[\href{https://stacks.math.columbia.edu/tag/066U}{Tag
    066U}]{stacks-project}. This proves (1).

    Part (2) follows by a straightforward argument by splitting the Postnikov
    tower of $\Fscr$. One can for instance adapt the proof of \cite[\href{https://stacks.math.columbia.edu/tag/0EWX}{Tag 0EWX}]{stacks-project} because for any $i, j$, 
    $$\Ext^2_{\Spec(V)}(H^i(\cF), H^j(\cF)) = 0.$$ 
    The vanishing of the $\Ext^2$ modules follows by the equivalence $\Dscr_\qc(\Spec(V)) \simeq \Dscr(V)$ 
    and the fact that the cohomology modules are sheaves associated to finitely presented $V$-modules which have projective dimension at most $1$ by the earlier part of the corollary. 
    \end{proof}

\subsection{A finiteness result of Nagata}

Since valuation rings are not necessarily Noetherian, finitely generated
algebras over such rings are not necessarily finite presented. However, Nagata
showed the surprising fact that flatness along with finite generation implies
finite presentation. Even more surprisingly, this was later generalized to 
finitely generated flat algebras over arbitrary domains in \cite{raynaud-gruson}
(see Remark \ref{rem:ray-gru-remarkable}).

\begin{theorem}
\label{thm:Nagata-amazing}
Let $V$ be a valuation ring and $B$ be a finitely generated $V$-algebra. 
\begin{enumerate}
    \item[{\em (1)}] If $B$ is a torsion free (equivalently, flat) $V$-module, then $B$ is a finitely presented $V$-algebra.
    \item[{\em (2)}] If $M$ is a finitely generated $B$-module that is torsion free as a $V$-module, then $M$ is a finitely presented $B$-module.
\end{enumerate}
\end{theorem}

\begin{proof}
Part (1) is precisely \cite[Thm. 3]{nagata}, while part (2) follows from \cite[\href{https://stacks.math.columbia.edu/tag/053E}{Tag 053E}]{stacks-project} by reducing to the graded case.
\end{proof}

Thus, if $f: X \rightarrow \Spec(V)$ is flat and locally of finite type, where
$V$ is a valuation ring, then $f$ is automatically locally of finite
presentation. Theorem \ref{thm:Nagata-amazing} has many interesting
consequences. For example, it can be used to show that valuation rings are
universally cohesive.

\begin{corollary}
\label{cor:valuation-rings-universally-cohesive}
A valuation ring $V$ is universally cohesive, that is, if $X \rightarrow \Spec(V)$ is locally of finite presentation, then $\cO_X$ is coherent.

In particular, if $f: X \rightarrow \Spec(V)$ is a proper, finitely presented morphism, then for any $F \in \Dscr^b_{\coh}(X)$, $\bR f_* F$ is a perfect complex.
\end{corollary}

\begin{proof}
Since coherence of $\cO_X$ can be checked affine locally on $X$, we may assume
    that $X$ is affine. Thus, it suffices to show that any finitely presented
    algebra over $V$ is a coherent ring. By Lemma \ref{lem:polynomial-extensions}, 
    it suffices to show that a polynomial algebra
    over $V$ is coherent. So let $B = V[x_1,\dots,x_n]$ be a polynomial ring.
    Let $I$ be a finitely generated ideal of $B$. We need to show that $I$ is
    finitely presented as a $B$-module. But, $B$ is a torsion free $V$-module,
    hence so is $I$. Thus, $I$ is a finitely presented $B$-module by Theorem
    \ref{thm:Nagata-amazing}(2).

If $f: X \rightarrow \Spec(V)$ is proper and finitely presented, then for $F \in \Dscr^b_{\coh}(X)$, $\bR f_*F \in \Dscr^b_{\coh}(\Spec(V))$ by Fujiwara-Kato's finiteness result (Theorem \ref{thm:Fujiwara-Kato-finiteness}). But any element of $\Dscr^b_{\coh}(\Spec(V))$ is a perfect complex by Corollary \ref{cor:fp-modules-pd-1}.
\end{proof}

\begin{remark}
    A similar proof of coherence of polynomial algebras over valuation rings also
    appears in \cite[Thm. 7.3.3]{Gla89}. Note that Fujiwara and Kato only show
    universal cohesiveness of valuation rings which are $f$-adic (aka
    microbial) in the valuation topology and complete with respect to an
    element in the height $1$ prime \cite[Chap. 0, Cor. 9.2.8]{FK18}. As such,
    Fujiwara and Kato do not observe Corollary
    \ref{cor:valuation-rings-universally-cohesive} for arbitrary valuation
    rings. Matthew Morrow observed this corollary to the first author, but with
    a different proof involving Raynaud--Gruson's platification par
    \'eclatement~\cite{raynaud-gruson}. 
\end{remark}

Theorem \ref{thm:Nagata-amazing} allows us to replace a proper, 
surjective morphism over the spectrum of a valuation ring by a flat, 
surjective, proper and finitely presented morphism in our investigation 
of the derived splinter condition. This is highlighted in the next result.

\begin{corollary}
\label{cor:dominating-proper-by-flat}
Let $V$ be a valuation ring and $f: X \rightarrow \Spec(V)$ be locally of finite type such that $f$ is closed (as a map of topological spaces) and surjective. Then there exists a closed subscheme $X' \subseteq X$ such that the composition 
$$X' \hookrightarrow X \xrightarrow{f} \Spec(V)$$ 
is flat, surjective (i.e. faithfully flat) and locally of finite presentation.
    In particular, if $f$ is
\begin{enumerate}
    \item[{\em (1)}] quasi-compact,
    \item[{\em (2)}] quasi-separated,
    \item[{\em (3)}] separated,
    \item[{\em (4)}] universally closed, or
    \item[{\em (5)}] proper,
\end{enumerate} 
then so is $X' \rightarrow \Spec(V)$.
\end{corollary}

\begin{proof}
Since $f: X \rightarrow \Spec(V)$ is surjective, the induced map on global sections
\[
f^{\sharp}: V \rightarrow \Gamma(X, \cO_X)
\]
is injective because $V$ is a domain and the image of $f$ is contained in
    $\bV(\ker(f^\sharp)) \subseteq \Spec(V)$. Let $\mathfrak{m}$ be the maximal
    ideal of $V$. For any non-zero $v \in \mathfrak{m}$, let
    $\cI_v \coloneqq \ker(\cO_X \xrightarrow{f^\sharp v \cdot} \cO_X)$. Then
    each $\cI_v$ is a proper ($\cI_v \neq \cO_X$ by injectivity of $f^\sharp$),
    quasi-coherent ideal of $X$. Define
\[
    \cI \coloneqq \sum_{v \in \mathfrak{m} - \{0\}} \cI_v,
\]
which is quasi-coherent by \cite[Cor. 7.19(3)]{GW10}.  
Note that if $\Spec(B)$ is an 
affine open subscheme of $X$, then $\cI(\Spec(B))$ is precisely the $V$-torsion 
submodule of $B$ (that also happens to be an ideal of $B$). 
The surjectivity of $f$ implies $\cI \neq \cO_X$, since any affine open neighborhood 
of a point
of $X$ lying above the generic point of $\Spec(V)$ cannot be $V$-torsion. 
Let $X'$ be the closed 
subscheme of $X$ defined by  $\cI$. Then by construction, the affine open subschemes 
of $X'$ are torsion free, hence flat over $V$ (Lemma \ref{lem:fin-gen-torsion free}). 
In particular, the composition 
$$f' \coloneqq X' \subseteq X \xrightarrow{f} \Spec(V)$$ 
is flat and locally of finite type, therefore also locally of finite presentation 
by Nagata's result (Theorem \ref{thm:Nagata-amazing}). Thus, $f'$ is an open map
\cite[\href{https://stacks.math.columbia.edu/tag/01UA}{Tag 01UA}]{stacks-project}
as well as a closed map (being a composition of two closed maps). Then
$f'(X')$ is a non-empty clopen subset of the connected space $\Spec(V)$, and so,
must equal all of $\Spec(V)$. That is, $f'$ is surjective.

If $f$ satisfies any of the properties (1)-(5), then so does $f'$ because closed 
embeddings satisfy (1)-(5), and these properties are preserved under composition.
\end{proof}

\section{Some classes of ind-regular valuation rings}

Classical local uniformization in characteristic $0$ implies that
characteristic $0$ valuation rings are filtered colimits of smooth
$\QQ$-subalgebras, that is, they are ind-smooth. We will show how work of de
Jong and Temkin implies similar results for two important classes of valuation
rings in prime and mixed characteristics. In fact, we prove that in these cases the
valuation rings in question are filtered colimits of their finitely generated (over
$\ZZ_{(p)}$) regular subalgebras.

\subsection{Perfect valuation rings are ind-smooth}

We first give a proof of the fact that every perfect valuation ring of
characteristic $p$ is ind-smooth, that is, a filtered colimit of smooth
$\FF_p$-algebras. This is a consequence of Temkin's inseparable local uniformization
theorem.

\begin{proposition}\label{prop:indsmooth}
    Every perfect valuation ring of characteristic $p > 0$ is
    a directed colimit of its smooth $\FF_p$-subalgebras.
\end{proposition}

\begin{proof}
    Let $K^\circ$ be a perfect valuation ring of characteristic $p>0$ 
    with field of fractions $K$. Note that $K$ is a perfect field. 
    We want to prove
    that the category of smooth
    $\FF_p$-subalgebras of $K^\circ$ with inclusions for transition maps
    is filtering with $K^\circ$ as colimit. We know that the
    category of finitely generated
    $\FF_p$-subalgebras of $K^\circ$ with inclusions for transition maps
    is filtering and $K^\circ$ is the filtered colimit of this system. 
    Thus, it is enough to check that if $R$ is a finitely
    generated $\FF_p$-subalgebra of $K^\circ$, then there is a
    factorization $R\hookrightarrow T\hookrightarrow K^\circ$ where $T$ is a smooth
    $\FF_p$-algebra; see for example~\cite[\href{https://stacks.math.columbia.edu/tag/0BUC}{Tag~0BUC}]{stacks-project}.
    Note that $R$ is a domain since $K^\circ$ is. Let $M$ denote the fraction
    field of $R$ and $M^\circ = M \cap K^\circ$ be the restriction of $K^\circ$ to $M$. 
    Since $M$ is a finitely generated field extension of $\FF_p$,
    by Temkin's inseparable local uniformization ~\cite[Thm.~1.3.2]{temkin-inseparable},
    there is a finite purely inseparable extension $L$ of $M$ and a
    subextension $R\subseteq S\subseteq M^\circ$ with $S$ a finitely generated
    $\FF_p$-algebra and such that if $L^\circ$ is the unique extension of
    $M^\circ$ to $L$, then $L^\circ$ is centered on a $\FF_p$-smooth point of
    the normalization of $S$ in $L$. In particular, there is an
    extension $S\rightarrow T$ such that $T$ is $\FF_p$-smooth with fraction field $L$
    and we have a
    commutative diagram 
    $$\xymatrix{
        R\ar@{^{(}->}[r]\ar@{^{(}->}[d]&M^\circ\ar@{^{(}->}[d]\ar[r]& K^\circ\\
        T\ar@{^{(}->}[r]&L^\circ.\ar@{.>}[ur]&
    }$$ 
    To see the existence of the dotted arrow, first note that $L$, 
    being a purely inseparable extension of $M$, 
    embeds in the perfect field $K$. Since 
    $L^\circ$ and $L \cap K^\circ$ are both valuation rings of $L$ dominating $M^\circ$, 
    we then get $L^\circ = L \cap K^\circ$ by uniqueness of the extension of
    $M^\circ$ to $L$. Thus, $K^\circ$ dominates $L^\circ$, and consequently, we have
    inclusions $R \hookrightarrow T \hookrightarrow K^\circ$ where $T$ is $\FF_p$-smooth.
\end{proof}

\subsection{Valuation rings of $p$-separably closed fields are ind-regular}

Fix a prime $p > 0$.
We will say that a field $K$ is \emph{$p$-separably closed}
if for each $n>0$, every finite \'etale $K$-algebra $L$ of 
degree\footnote{The \emph{degree} of a finite algebra over a field is its vector
space dimension over the field.} $p^n$ has a
section $L\rightarrow K$ in the category of $K$-algebras (that is, a $K$-algebra map such that the composition
$K\rightarrow L\rightarrow K$ is the identity of $K$).\footnote{This is not
standard terminology, but it is the notion we need. More common, for
example as in~\cite[Sec. VI.1]{NSW}, is the notion
of $p$-closed field
where it is required that $K$ have no Galois extensions of degree $p$ (or
equivalently of degree $p^n$ for $n>0$). A $p$-separably closed field is
$p$-closed since non-trivial field extensions do not admit sections, but the
converse is not generally true. For example, let $\FF_\ell$ be a finite field
and fix two distinct primes $q_1$ and $q_2$ which are both different from
$p$. We let $G\iso\prod_{r}\ZZ_r$ be the absolute Galois group of $\FF_\ell$, where the product ranges over all primes
$r$, and we let $G'\subseteq
G$ be the subgroup $\prod_{r\neq q_1,q_2}\ZZ_r$. Set $K=\overline{\FF}_\ell^{G'}$.
Then, $K$ is evidently $p$-closed since its Galois group is
$\ZZ_{q_1}\times\ZZ_{q_2}$, but it is not $p$-separably closed since we can
solve the equation $mq_1+nq_2=p^a$ for some $m,n,a>0$.}

For example, any separably closed field is $p$-separably closed for
any prime $p$. The field
$\RR$ is $p$-separably closed for any odd prime $p$. Given any field $k$ with
separable Galois group $G$, fix a prime $\ell\neq p$ and consider an
$\ell$-Sylow subgroup $G_\ell\subseteq G$ corresponding to a maximal
prime-to-$\ell$-extension $k\rightarrow K$. See~\cite[Sec. I.1.4]{serre-galois} for the
fact that such $K$ exist. The field $K$ is $p$-separably
closed. Indeed, any separable extension of $K$ splits as a product of
$\ell$-power degree field extensions of $K$. The only way the sum of these
degrees can be $p^n$ is if at least one has degree $1$.

The next result was communicated to us by Elden Elmanto and Marc Hoyois.

\begin{proposition}\label{prop:indregular}
    Let $K$ be a $p$-separably closed field. If $K^\circ$ is a valuation ring of $K$
    of residue characteristic $p$, then $K^\circ$ is a directed colimit of its regular
    finitely generated $\ZZ_{(p)}$-subalgebras.
\end{proposition}

The main ingredient in the proof of Proposition \ref{prop:indregular} is the following
strengthening of de Jong's work on alterations \cite[Thm. 4.1]{AJ96} due to
Temkin~\cite[Thm. 1.2.5]{temkin-tame}; see also~\cite{illusie-temkin}.

\begin{theorem}[\cite{temkin-tame}]\label{thm:Temkin-alteration}
Let $X$ be an integral scheme with a nowhere dense closed subset Z,
and assume $X$ admits a morphism of finite type to a quasi-excellent
scheme $S$ with $\dim(S) \leq 3$. Then there is a projective alteration 
$$b: X' \rightarrow X$$
with regular integral source such that $b^{-1}(Z)$ is a snc divisor and the
degree of the extension of function fields $[k(X'):k(X)]$ has prime factors among
the set of primes which are the characteristics of the residue fields of $X$.
Moreover, if $S = \Spec(k)$ for a perfect field $k$, then the alteration
may be chosen to be generically \'etale.
\end{theorem}

\begin{remark}
We will only apply Temkin's result when the base 
$S$ is the spectrum of a field or an excellent discrete valuation ring.
\end{remark}

\begin{proof}[Proof of Proposition \ref{prop:indregular}]
    The argument is similar to the proof of Proposition~\ref{prop:indsmooth},
    but we use Temkin's work on $p$-alterations~\cite{temkin-tame} instead of
    his work on inseparable local uniformization.
    
    Note that $K^\circ$ is a $\ZZ_{(p)}$-algebra since it has residue characteristic $p$.
    As in the proof of
    Proposition~\ref{prop:indsmooth}, it is enough to prove the following: for
    every finitely generated $\ZZ_{(p)}$-subalgebra $R$ of $K^\circ$, 
    there is a regular finitely generated
    $\ZZ_{(p)}$-algebra $T$ together with a factorization of $R\hookrightarrow
    K^\circ$ through maps $R\hookrightarrow T\hookrightarrow K^\circ$. 
    We may assume $R$ is not regular as otherwise there is nothing to prove.

    Let $M\subseteq K$ be the field of fractions of $R$ and let $M^\circ=M\cap K^\circ$.
    By Theorem \ref{thm:Temkin-alteration}, there is a regular integral scheme $X'$ and a
    projective alteration 
    $$X'\rightarrow\Spec(R)$$ 
    of degree $p^n$ for some $n \geq 1$
    (since we assumed $R$ is not regular and since all residue fields of $R$ have 
    characteristic $0$ or $p$). Moreover, since $\FF_p$ is perfect, we can
    also assume regardless of the characteristic of $K$ that $X'\rightarrow\Spec(R)$
    is generically \'etale.
        
     Let $L$ be the function field of $X'$. Since
    $K$ is $p$-separably closed, and $M \hookrightarrow L$ is a finite separable field 
    extension of degree $p^n$, there is a dotted arrow making the diagram
    $$\xymatrix{
        L\ar@{.>}[dr]&\\
        M\ar[u]\ar[r]&K
    }$$ commute. Indeed, by definition of $p$-separable closedness,
    $K\rightarrow L\otimes_M
    K$ splits in the category of $K$-algebras. 
    In particular, we obtain a commutative diagram
    $$\xymatrix{
        \Spec K\ar[r]\ar[d]&X'\ar[d]\\
        \Spec K^\circ\ar[r]\ar@{.>}[ur]&\Spec R
    }$$ together with a lift by the valuative criterion.
    Let $\Spec T\subseteq X'$ be an open neighborhood of the image under $\Spec
    K^\circ\rightarrow X'$ of the
    closed point of $\Spec K^\circ$. We then have a commutative
    diagram $$\xymatrix{
        R\ar@{^{(}->}[d]\ar[r]& K^\circ\ar[d]\\
        T\ar@{^{(}->}[r]\ar[ur]&K
    }$$
    where $T$ is regular Noetherian, finite type over $\ZZ_{(p)}$, 
    and injects into $K^\circ$ by commutativity
    of the above square. But this is what we wanted to show.
\end{proof}

As a consequence, we can deduce that absolutely integrally closed valuation rings
are always ind-regular.

\begin{corollary}
\label{cor:absoluteintegralclosed}
Let $V$ be an absolutely integrally closed valuation ring that dominates the 
local ring $S$, where $S = \QQ, \FF_p$ or $\ZZ_{(p)}$ for a prime $p > 0$. 
Then $V$ is a directed colimit of its finitely generated and regular $S$-subalgebras.
In particular, the absolute integral closure of a Henselian valuation ring is 
ind-regular.
\end{corollary}

\begin{proof}
Note the fraction field $K$ of $V$ is algebraically closed. In particular, 
$K$ is perfect and $p$-separably closed. If $S = \QQ$, that is, when $V$ 
has equal characteristic $0$, the result follows by Zariski's
local uniformization \cite{Zar40} (using a method of proof similar to 
Proposition~\ref{prop:indsmooth}) because \emph{any} valuation ring
of characteristic $0$ is a colimit of its smooth $\QQ$-subalgebras. 
In the other cases, the assertion follows by 
Proposition \ref{prop:indsmooth} (when $S = \FF_p$) and 
Proposition \ref{prop:indregular} (when $S = \ZZ_{(p)}$).

Finally, the absolute integral closure of a Henselian valuation
ring is also a valuation ring -- it is local by \cite[Prop. 1.4]{Art71}
and hence a valuation ring by \cite[Chp.~VI, $\mathsection 8.6$, Prop.~6]{Bou89},
and so, we are done by the first part of this Corollary.
\end{proof}

\begin{remark}
\label{rem:not-indsmooth}
{\*}
\begin{enumerate}
    \item[(1)] One can use de Jong's alterations \cite{AJ96} 
instead of Temkin's p-alterations
to show that absolutely integrally closed valuation rings are directed
colimits of regular subalgebras by adapting the argument of
Proposition \ref{prop:indregular}. Indeed, in the proof of loc. cit.
if $K$ is algebraically closed and contains the field $M$, then any 
finite extension of $M$ (coming from the function field of a regular alteration
$X'$) will automatically embed in $K$.
This then allows one to find a center of $K^\circ$ on $X'$ using the
valuative criterion, and the rest of the argument applies verbatim.

 \item[(2)] The fact that absolutely integrally closed valuation rings are filtered colimits
of their finitely generated (over $\QQ, \FF_p$ or $\ZZ_{(p)}$) regular subalgebras is not
a property shared by other absolutely integrally closed domains. For example, let
$R$ be a finitely generated $\QQ$-algebra which is a normal domain that is 
not Cohen--Macaulay (thus, $\dim R \geq 3$). We claim that the absolute integral closure
$R^+$ of $R$ cannot be expressed as a filtered colimit of  
regular $\ZZ$-algebras, that is, $R^+$ is not ind-regular. Indeed, assume for contradiction
that $R^+ = \colim_I S_i$ where the $S_i$ are regular $\ZZ$-algebras. Since $R^+$
is a $\QQ$-algebra and regularity is preserved under localization, replacing $S_i$ by
$\QQ \otimes_{\ZZ} S_i$, we may assume each $S_i$ is a regular $\QQ$-algebra.
Now as $R$ is finitely presented over $\QQ$, there exists some $S_i$ such that the $\QQ$-algebra map
$R \hookrightarrow R^+$ factors as $R \rightarrow S_i \rightarrow R^+$
   \cite[\href{https://stacks.math.columbia.edu/tag/00QO}{Tag
    00QO}]{stacks-project}.
Note $S_i \rightarrow R^+$ need not be injective because we are not assuming that $R^+$
has to be a colimit of regular \emph{subalgebras}. 
By normality of $R$, a well-known argument involving the trace map shows that
$R \hookrightarrow R^+$ is a pure map \cite[Ex. 1.1]{Bha12}, that is, the composition
$R \rightarrow S_i \rightarrow R^+$ is pure. Then $R \rightarrow S_i$ is pure 
as well. Since a pure subring of an equicharacteristic regular ring is Cohen--Macaulay 
\cite[Thm.~2.3]{HH95} (see also \cite{HM18} for a mixed characteristic version),
we conclude that $R$ must be Cohen--Macaulay, contradicting our choice. 
In a certain sense $R^+$ is as nice as possible. Indeed, choosing
a Noether Normalization $\QQ[x_1,\dots,x_n]$ of $R$, it is not hard to see that $R^+$
can be identified with $\QQ[x_1,\dots,x_n]^+$, the absolute integral closure of a 
polynomial ring over $\QQ$.

\item[(3)] With a little more work, 
similar examples of the failure of ind-regularity of $R^+$ 
can also be given in prime characteristic in such a way that the ring 
$R$ has `mild' singularities.  Indeed, the argument in (2) shows that if $R$ is 
a domain of prime characteristic $p > 0$ that is finite type over $\FF_p$, if
the map $R \rightarrow R^+$ is pure, and if $R^+$ is a filtered colimit of regular
$\FF_p$-algebras, then $R$ will be a pure subring of 
a regular ring.
Note that purity of $R \rightarrow R^+$ is equivalent to saying that $R$ is 
a splinter (Definition \ref{def:splinters} and Proposition 
\ref{prop:decent-colimit-splinter}), which by a result of Smith
implies that $R$ is $F$-rational \cite{Smi94} (a prime characteristic
analogue of rational singularities) and hence Cohen--Macaulay \cite{HH94}.
Thus, unlike our example in (2), Cohen--Macaulayness is automatic in prime
characteristic with purity of $R \rightarrow R^+$, when $R$ is excellent.
A well-known result of Huneke and Sharp shows that local cohomology modules 
of regular rings
of prime characteristic with support in any ideal have finitely many 
associated primes \cite[Cor. 2.3]{HS93}, and this property also descends 
under pure maps\footnote{We thank Linquan Ma for pointing out to us that the 
property of local cohomology 
having finitely many associated primes descends under pure maps.
} by an observation of N\'u\~{n}ez-Betancourt \cite[Prop. 2.2]{N-B12}.
The upshot is that if $R^+$ is ind-regular over $\FF_p$ and $R \rightarrow R^+$ is pure, 
then all local cohomology
modules of $R$ will have finitely many associated primes.
However, Singh and Swanson show that the ring
$$R = \FF_p[s,t,u,v,w,x,y,z]/(su^2x^2 + sv^2y^2 + tuxvy + tw^2z^2)$$
is one for which $R \rightarrow R^+$ is pure, but where the local cohomology module
$H^3_{(x,y,z)}(R)$ has infinitely many associated primes \cite[Thm. 5.1]{SS04}. Thus, such an $R$
cannot have $R^+$ that is ind-regular over $\FF_p$ by our discussion above. 
We note that in the theory of Frobenius singularities, the ring $R$ in 
Singh and Swanson's example is considered to have `mild' singularities since it is 
strongly $F$-regular, a prime characteristic analogue of a klt-singularity.
Note that any $F$-regular domain $R$ has the property that $R \rightarrow R^+$
is pure.
\end{enumerate}
\end{remark}

\section{Valuation rings are derived splinters}
Utilizing the results of the previous sections, we provide three different proofs of the following
result.

\begin{theorem}
\label{thm:valuations-derived-splinters}
Let $V$ be a valuation ring. If $f: X \rightarrow \Spec(V)$ is a proper surjective morphism, then
the induced map $\cO_{\Spec(V)} \rightarrow \bR f_*\cO_{X}$ splits in $\Dscr_\qc(\Spec(V))$. In 
particular, valuation rings are derived splinters.
\end{theorem}

\begin{remark}
\label{rem:flat-fp}
Note that we do not assume $f$ is finitely presented in the statement of
Theorem \ref{thm:valuations-derived-splinters}, although we have finite 
presentation hypothesis built in to our definition of a derived splinter 
(Definition \ref{def:d-splinter}). 
Using Corollary \ref{cor:dominating-proper-by-flat}, 
it suffices to prove Theorem \ref{thm:valuations-derived-splinters} when $f$ 
is flat and finitely presented, 
in addition to being proper and surjective. Indeed, for an arbitrary proper 
surjective $f: X \rightarrow \Spec(V)$, choose a closed subscheme 
$i: X' \hookrightarrow X$ of $X$ such that 
$f \circ i$ is proper, flat, surjective and finitely presented as in 
Corollary \ref{cor:dominating-proper-by-flat}. Since
$$\cO_{\Spec(V)} \rightarrow \bR (f \circ i)_* \cO_{X'}$$
factors as $\cO_{\Spec(V)} \rightarrow \bR f_* \cO_{X} \rightarrow \bR f_* \bR i_* \cO_{X'}$ 
(the second map is obtained by applying $\bR f_*$ to $\cO_{X} \rightarrow \bR i_* \cO_{X'}$), 
a splitting of $\cO_{\Spec(V)} \rightarrow \bR (f \circ i)_* \cO_{X'}$ 
will induce a splitting of $\cO_{\Spec(V)} \rightarrow \bR f_* \cO_{X}$.
\end{remark}

Recall that a domain $R$ is a \emph{Pr\"ufer domain} if for any prime ideal $\fp$ of $R$
the localization $R_\fp$ is a valuation ring.  As a consequence of Theorem 
\ref{thm:valuations-derived-splinters}, we can conclude that Pr\"ufer domains are also 
derived splinters.

\begin{corollary}
\label{cor:Prufer}
Any Pr\"ufer domain is a universally cohesive derived splinter. In particular,
the absolute integral closure of a valuation ring is a derived splinter.
\end{corollary}

\begin{proof}
Let $R$ be a Pr\"ufer domain. Since each local ring of $R$ is a valuation domain,
the fact that $R$ is a derived splinter will follow from Theorem 
\ref{thm:valuations-derived-splinters} and Lemma \ref{lem:d-splinter-local}
if we can show that $R$ is universally cohesive.\footnote{Universal cohesiveness
of a Pr\"ufer domain does not follow from universal cohesiveness of all its 
localizations by Harris's example of a ring which is not coherent even though 
its local rings are all Noetherian \cite[Thm. 3]{Har67}.}

First note that since flatness is a local property, any torsion free 
module over $R$ is flat by the correspond fact for valuation rings 
(Lemma \ref{lem:fin-gen-torsion free}). To show that $R$ is 
universally cohesive, it suffices to show by Lemma \ref{lem:polynomial-extensions} 
that a polynomial ring 
$B = R[x_1,...,x_n]$ over $R$ is coherent. That is, we have to show that
any finitely generated ideal $I$ of $B$ is finitely presented as a $B$-module.
Note that $I$ is a torsion free, and consequently, a flat $R$-module and that
$B$ is a finitely presented $R$-algebra. Then
$I$ is a finitely presented $B$-module by a miraculous result of
Raynaud and Gruson \cite[Thm.~(3.4.6)]
{raynaud-gruson}.

If $V$ is a valuation ring, then its absolute integral closure 
$V^+$ is a Pr\"ufer domain \cite[Chp.~VI, $\mathsection 8.6$, Prop.~6]{Bou89}, 
but not necessarily a valuation ring unless $V$ is Henselian. 
Then the result follows.
\end{proof}

\begin{remark}
\label{rem:ray-gru-remarkable}
A straightforward
consequence of Raynaud and Gruson's \cite[Thm. (3.4.6)]{raynaud-gruson} provides a
significant generalization
of Nagata's result recalled in Theorem \ref{thm:Nagata-amazing}(1). Namely, if $A$ is 
\emph{any} integral domain  
and $B$ is a finitely generated $A$-algebra that is flat
over $A$, then $B$ is a finitely presented $A$-algebra. This last 
observation shows that a Pr\"ufer domain $R$ is a derived splinter in 
Bhatt's sense of the notion, that is, for any proper, surjective
morphism $f: X \rightarrow \Spec(R)$, $\cO_{\Spec(R)} \rightarrow \bR f_*\cO_X$
splits even when $f$ is not finitely presented. Indeed, since 
torsion free $R$-modules are flat, the closed subscheme $X'$ of $X$ defined
by the $R$-torsion ideal is torsion free, hence flat, and hence finitely presented
over $\Spec(R)$. Moreover, $X' \rightarrow \Spec(R)$ is surjective and proper
as in the proof of Corollary \ref{cor:dominating-proper-by-flat}. Thus, the 
splitting of $\cO_{\Spec(R)} \rightarrow \bR (f\circ i)_* \cO_{X'}$ (by Corollary \ref{cor:Prufer}) 
induces 
a splitting of $\cO_{\Spec(R)} \rightarrow \bR f_*\cO_X$ by Remark \ref{rem:flat-fp}.
\end{remark}

\subsection{A proof using the valuative criterion}
Recall that a quasi-compact universally closed morphism of schemes $f: X \rightarrow Y$ 
satisfies the \emph{existence} part of the valuative criterion, that is, 
given any solid commutative diagram
\begin{center}
\begin{tikzcd}
  \Spec(K) \arrow[r] \arrow[d]
    & X \arrow[d, "f"] \\
  \Spec(V) \arrow[ru , dashrightarrow]\arrow[r]&Y 
\end{tikzcd}
\end{center}
where $V$ is a valuation ring with fraction field $K$, the dotted arrow exists making the 
above diagram commute \cite[\href{https://stacks.math.columbia.edu/tag/01KF}{Tag 01KF}]
{stacks-project}. 

Using the existence part of the valuative criterion, one can show that 
absolutely integrally closed valuation rings satisfy the derived splinter 
property for a more general class of morphisms
than just finitely presented, proper, surjective ones.

\begin{lemma}
\label{lem:lifting-identity}
Let $V$ be an absolutely integrally closed valuation ring. 
Let $f: X \rightarrow \Spec(V)$ be a 
quasi-compact universally closed surjective morphism of schemes 
whose generic fiber is locally of 
finite type. Then $f$ admits a section, that is, there exists
a morphism 
$g: \Spec(V) \rightarrow X$ such that $f \circ g = \id_{\Spec(V)}$. 
In particular,
$\cO_{\Spec(V)} \rightarrow \bR f_*\cO_X$ splits in $\Dscr(\Spec(V))$. 
\end{lemma}

\begin{proof}
If a section $g$ of $f$ exists, then a section of 
$\cO_{\Spec(V)} \rightarrow \bR f_*\cO_X$ is 
obtained by applying $\bR f_*$ to $\cO_X \rightarrow \bR g_*\cO_{\Spec(V)}$ and 
using the natural isomorphism $\bR(f \circ g)_* \simeq \bR f_* \circ \bR g_* $. 

To see the existence of $g$, note that if $K = \Frac(V)$, then the generic fiber 
$$f_K: X_K \rightarrow \Spec(K)$$ 
of $f$ has a $K$-rational point. Indeed, $K$ is algebraically closed by the 
assumption that
$V$ is absolutely integrally closed and $X_K$ is of locally of finite type over 
$\Spec(K)$ and non-empty. The composition
$$\varphi \coloneqq \Spec(K) \rightarrow X_K \rightarrow X$$
then gives us a commutative diagram
\begin{center}
\begin{tikzcd}
  \Spec(K) \arrow[r, "\varphi"] \arrow[d]
    & X \arrow[d, "f"] \\
  \Spec(V) \arrow[r, "\id"]&\Spec(V),
\end{tikzcd}
\end{center}
and so, the existence part of the valuative criterion provides a lift of $\id_{\Spec(V)}$ along $f$,
giving us the desired section $g$.
\end{proof}

We can now provide one proof of Theorem \ref{thm:valuations-derived-splinters}.

\begin{proof}[First proof of Theorem \ref{thm:valuations-derived-splinters}]
Let $f: X \rightarrow \Spec(V)$ be a proper surjective morphism. Using Remark \ref{rem:flat-fp}, we
may assume $f$ is finitely presented. Let $K$ be the fraction field of $V$ and
    let $\overline{K}$ 
denote an algebraic closure of $K$. Choose a valuation ring $\overline{V}$ of $\overline{K}$ that 
dominates $V$. Then the local homomorphism $V \hookrightarrow \overline{V}$ is flat (hence faithfully
flat) by Lemma \ref{lem:fin-gen-torsion free}. Moreover, $\overline{V}$ is absolutely integrally
closed because $\overline{K}$ is algebraically closed and $\overline{V}$ is integrally closed in
$\overline{K}$. Since $\overline{V}$ is a derived splinter by Lemma \ref{lem:lifting-identity}, it
follows that $\cO_{\Spec(V)} \rightarrow \bR f_* \cO_X$ splits by faithfully flat descent of the 
derived splinter property for universally cohesive rings (Proposition 
\ref{prop:descent-derived-splinter}). Note that in order to apply Proposition 
\ref{prop:descent-derived-splinter} we need $f$ to be finitely presented in addition to being proper
and surjective.
\end{proof}

\subsection{A proof using the splinter property}

Just as we modified the usual definition of a derived splinter to make the notion
more amenable to explorations in a non-Noetherian setting (see Definition \ref{def:d-splinter}),
we would like to make a similar modification of the notion of a splinter.
To justify our definition, we first prove the following algebraic result:

\begin{lemma}
\label{lem:splinter-new-def}
Let $R$ be a ring and consider the following statements.
\begin{enumerate}
    \item[{\em (i)}] Every finite ring map $R \rightarrow S$ that induces a surjection on $\Spec$
	      admits an $R$-linear left inverse.
      \item[{\em (ii)}] Every finite ring map $R \rightarrow S$ that induces a surjection on $\Spec$
		  is pure as a map of $R$-modules.
      \item[{\em (iii)}] Every ring map $R \rightarrow S$ that induces a surjection on $\Spec$ and 
          such that $S$ is finitely presented as an $R$-module admits an $R$-linear left 
          inverse.
\end{enumerate}
Statements (ii) and (iii) are equivalent, while (i) implies (ii) (hence (iii)). Moreover, all three
conditions are equivalent when $R$ is Noetherian.
\end{lemma}

\begin{proof}
It is clear that (i) implies (ii) and (iii). We now prove the equivalence of (ii) and (iii).
Assuming (ii), if $R \rightarrow S$ is map which is surjective on $\Spec$ and such 
that $S$ is a finitely presented $R$-module, then $R \rightarrow S$ is pure. In particular,
$R \rightarrow S$ is injective, and so, $\coker(R \rightarrow S)$ is a finitely presented 
$R$-module by \cite[\href{https://stacks.math.columbia.edu/tag/0519}{Tag 0519}(4)]{stacks-project}.
However, any pure ring map with a finitely presented cokernel splits by \cite[Cor. 5.6]{HR76},
proving (ii) $\Rightarrow$ (iii). 

Now assume (iii). Let $R \rightarrow S$ be a finite ring map that is surjective on $\Spec$. 
By \cite[\href{https://stacks.math.columbia.edu/tag/09YY}{Tag 09YY}]{stacks-project},
$S$ can be expressed as a directed colimit of a collection of $R$-algebras $\{S_i: i \in I\}$,
where each $S_i$ is a finite and finitely presented $R$-algebra and the transition maps
$S_i \rightarrow S_j$ are surjections. Since the composition 
$$R \rightarrow S_i \rightarrow S = \colim S_i$$
induces a surjection on $\Spec$, the maps $R \rightarrow S_i$ do as well. But a finite
and finitely presented $R$-algebra is finitely presented as an $R$-module 
\cite[\href{https://stacks.math.columbia.edu/tag/0564}{Tag 0564}]{stacks-project}. Therefore
each map $R \rightarrow S_i$ splits by (iii), and one can check that a directed colimit of
split maps is pure. Thus, (iii) $\Rightarrow$ (ii).

Finally, when $R$ is Noetherian we see that (iii) implies (i) because any finite $R$-module is 
also finitely presented, and so, all three conditions are equivalent.
\end{proof}

Inspired by the lemma, we make the following definition.

\begin{definition}
\label{def:splinters}
A ring $R$ is a \emph{splinter} if it satisfies condition (ii) (equivalently, (iii)) of Lemma 
\ref{lem:splinter-new-def}. Globally, a scheme $S$ is a \emph{splinter} if for every finite, 
surjective and finitely presented morphism $f: X \rightarrow S$, the induced map 
$\cO_S \rightarrow f_*\cO_X$ splits in the category of $\cO_S$-modules.
\end{definition}

Thus, any derived splinter (Definition \ref{def:d-splinter}) is a splinter.

\begin{remark}
Rings satisfying condition (i) of Lemma \ref{lem:splinter-new-def} are usually defined to be
splinters in the literature. However, purity is a less restrictive notion than 
splitting in the setting of non-Noetherian rings. Moreover, a pure ring map satisfies 
many of the desirable properties of split (and even faithfully
flat) ring maps. Globally, finite and finitely presented morphisms are more
    tractable than finite morphisms in a non-Noetherian setting.
\end{remark}

In \cite[Rem. 5.0.7]{Dat17(a)}, the second author observes that valuation rings of arbitrary 
characteristic satisfy condition the stronger condition (i) of Lemma \ref{lem:splinter-new-def}. 
Since the proof is a fairly straightforward consequence of the finiteness results 
discussed in this paper, we include the argument for the reader's convenience.

\begin{proposition}
\label{prop:valuations-splinters}
Let $V$ be a valuation ring. Then any finite ring map $V \rightarrow S$ that induces 
a surjection on $\Spec$ admits a $V$-linear section. In particular, valuation 
rings are splinters.
\end{proposition}

\begin{proof}
Let $V$ be a valuation ring and $V \rightarrow S$ be a finite map that induces a surjection 
on $\Spec$. Choose a prime ideal $\fp$ of $S$ that lies over the zero ideal of $V$. Then 
the composition $V \rightarrow S \rightarrow S/\fp$ is also finite and injective, thus 
surjective on $\Spec$. Since $V \rightarrow S$ splits if the composition 
$V \rightarrow S \rightarrow S/\fp$ does, it suffices to assume $S$ is a
    domain. In this case, $S$ 
is a torsion free $V$-module since $V \rightarrow S$ is an injection.
    Therefore, $S$ is a free
$V$-module by Lemma \ref{lem:fin-gen-torsion free}. Nakayama's Lemma now shows that $1 \in S$ 
must be part of a free $V$-basis of $S$, and so, $V \rightarrow S$ admits a splitting.
\end{proof}

We now provide another proof of Theorem \ref{thm:valuations-derived-splinters} using 
Proposition \ref{prop:valuations-splinters}.

\begin{proof}[Second proof of Theorem \ref{thm:valuations-derived-splinters}]
Again, as in Remark \ref{rem:flat-fp}, we may assume that $f: X \rightarrow \Spec(V)$ is 
finitely presented. Then by Corollary \ref{cor:valuation-rings-universally-cohesive}, $\cO_X$
is coherent and $\bR f_*\cO_X$ is a perfect complex. Consequently, by Corollary 
\ref{cor:fp-modules-pd-1},
\[
\bR f_*\cO_X \simeq \bigoplus_i R^if_*\cO_X[-i].
\]
Since $f: X \rightarrow \Spec(V)$ is surjective on $\Spec$, the ring homomorphism 
$V \rightarrow \Gamma(X,\cO_X)$ is injective, and moreover, 
$\Gamma(X,\cO_X)$ is a coherent $V$-module.
Thus, $V \rightarrow \Gamma(X,\cO_X)$ splits by Proposition \ref{prop:valuations-splinters}, inducing
a splitting of $\cO_{\Spec(V)} \rightarrow \bigoplus_i R^if_*\cO_X[-i]$.
\end{proof}

Defining splinters via the equivalent properties (ii) $\Leftrightarrow$ (iii) 
of Lemma \ref{lem:splinter-new-def} instead of (the stronger) property (i) 
also implies that the splinter condition behaves well with respect to descent 
and filtered colimits, as we now show.

\begin{proposition}
\label{prop:decent-colimit-splinter}
Let $R$ be any commutative ring.
\begin{enumerate}
    \item[{\em (i)}] If $R \rightarrow S$ is a pure ring map and $S$ is a splinter,
	      then $R$ is a splinter.
      \item[{\em (ii)}] Let $(I, \leq)$ be a directed set. If $R$ is the colimit 
      of a collection of rings $\{R_i\}_{i \in I}$ where each $R_i$ is a splinter, 
      then $R$ is a splinter.
      \item[{\em (iii)}] Let $R$ be a domain and $R^+$ denote its absolute integral
	closure. Then $R$ is a splinter if and only if the map $R \rightarrow R^+$
	is pure.
\end{enumerate}
\end{proposition}

\begin{proof}
To establish the splinter property,
it suffices to verify property (ii) or property (iii) from 
Lemma \ref{lem:splinter-new-def}.
Let $R \rightarrow T$ be a finite map that is surjective on $\Spec$.
Then $S \rightarrow T \otimes_R S$ is also a finite map that is surjective on 
$\Spec$, and so, it is pure. Using the cocartesian diagram
$$\xymatrix{
        R\ar[r]^{}\ar[d]^{}&S\ar[d]^{}\\
        T\ar[r]^{}&T \otimes_R S
    }$$
purity of $R \rightarrow T$ now follows because the composition 
    $R \rightarrow T \rightarrow T \otimes_R S$ is pure. This proves (i).

    For (ii), we have to show that if $f: R \rightarrow T$ is a ring map that
is surjective on $\Spec$ and $T$ is a finitely presented $R$-module, then $f$ 
admits an $R$-linear section. Note that $T$ is finitely presented
as an $R$-algebra \cite[\href{https://stacks.math.columbia.edu/tag/0564}{Tag 0564}]{stacks-project},
and so, there exists $i \in I$ and a finitely presented $R_i$-
algebra $T_i$ such that
$$T = T_{i} \otimes_{R_{i}} R.$$ 
Since $T$ is also a finite $R$-algebra, by 
\cite[\href{https://stacks.math.columbia.edu/tag/01ZO}{Tag 01O}]{stacks-project}
there exists $j \geq i$ such that 
$$T_j \coloneqq T_{i} \otimes_{R_i} R_j$$
is a finite $R_j$-algebra, hence also finitely presented as an $R_j$-module (again by
 \cite[\href{https://stacks.math.columbia.edu/tag/0564}{Tag 0564}]{stacks-project}). 
Moreover, we may further assume by 
\cite[\href{https://stacks.math.columbia.edu/tag/07RR}{Tag 07RR}]{stacks-project}
that $R_j \rightarrow T_j$ induces a surjection on $\Spec$. Since $R_j$ is a splinter,
it follows that $R_j \rightarrow T_j$ splits. Tensoring this splitting by
$-\otimes_{R_j} R$ then gives a splitting of $R \rightarrow T$, completing the 
    proof of (ii).

Suppose $R$ is a splinter. We can express $R^+$ as a filtered colimit 
of finitely generated $R$-subalgebras $R_i$. But $R \rightarrow R_i$ is a finite extension 
because $R \rightarrow R^+$ is integral, and hence pure because $R$ is a splinter. Since 
a filtered colimit of pure maps is pure, we see that $R \rightarrow R^+$ is pure. Conversely, suppose 
$R \rightarrow R^+$ is pure. Since $R^+$ is absolutely integrally closed, any finite extension 
$R^+ \hookrightarrow S$ admits a section that is, in fact, a ring homomorphism. 
To see this, choose a prime ideal $\fp$ of $S$ that contracts to the zero ideal of $R^+$. 
Then the composition $R^+ \rightarrow S \rightarrow S/\fp$ is again a finite ring 
extension. But now, since $S/\fp$ is a domain and $R^+$ is absolutely integrally 
closed, it follows that $R^+ \rightarrow S/\fp$ must be an isomorphism. Composing
$S \twoheadrightarrow S/\fp$ with the inverse of this isomorphism then gives
a left inverse of $R \rightarrow S$ in the category of $R$-algebras. In particular, $R^+$
is a splinter, allowing us to conclude that $R$ is a splinter by (i).
\end{proof}

\begin{remark}
{\*}
\begin{enumerate}
    \item[(1)] As an application of Proposition \ref{prop:decent-colimit-splinter}, if $R$ is a
splinter and a ring of prime characteristic $p > 0$, then its perfection $R_{\perf}$
is also a splinter. The converse holds if the Frobenius map of $R$ is pure. 

    \item[(2)] The proof of Proposition \ref{prop:decent-colimit-splinter}(ii) partially 
	globalizes -- if $S = \lim_I S_i$ is a limit of a directed inverse system of
	qcqs schemes with affine transition maps such that each $S_i$ is a splinter, then 
	$S$ is also a splinter. First, note that all the results on
	descending properties of morphisms under direct limits that are used in the proof 
	of loc. cit. continue to hold under the 
	above hypotheses on the direct system $\{S_i\}$. Second, if we have a Cartesian 
	diagram
	$$\xymatrix{
        X\ar[r]^{g'}\ar[d]^{f}&X_i\ar[d]^{f_i}\\
        S\ar[r]^{g}&S_i
    }$$
    where $f$ and $f_i$ are affine, then for any $F \in \QCoh(X_i)$, the (un-derived)
    base change map
    $\eta_F: g^*(f_i)_*F \rightarrow f_*(g')^*F$ is an isomorphism. Using this fact it 
    follows that if $\cO_{S_i} \rightarrow (f_i)_*\cO_{X_i}$ splits then so does
    $\cO_S \rightarrow f_*\cO_X$. The argument in the proof of
    Proposition \ref{prop:decent-colimit-splinter}(ii) now readily globalizes.
    
\item[(3)] For an alternate, but morally similar, proof of the splinter condition for
    a valuation ring $V$, note that the map $V \rightarrow V^+$ is faithfully flat since
    $V^+$ is a torsion free $V$-module (Lemma \ref{lem:fin-gen-torsion free}). 
    Thus $V$ is a splinter by Proposition \ref{prop:decent-colimit-splinter}(iii).
    \end{enumerate}
\end{remark}

\subsection{A proof using the derived direct summand theorem}

Recall that Hochster's direct summand conjecture, now a theorem due to Andr\'e 
\cite{And18} and Hochster \cite{Hoc73}, 
asserts that a Noetherian regular ring is a splinter. Unsurprisingly, 
the derived direct summand theorem is the derived analogue of the direct summand 
theorem. The statement, with an indication of proof, appears below:

\begin{theorem}
\label{thm:derived-direct-summand}
A Noetherian regular ring of arbitrary characteristic is a derived splinter.
\end{theorem}

\begin{proof}[Indication of proof]
A Noetherian regular ring of mixed characteristic is a derived splinter by 
\cite[Thm. 6.1]{Bha18}. Since the notions of splinters and derived splinters 
coincide for Noetherian schemes over a field of prime characteristic
by \cite[Thm. 1.4]{Bha12}, it follows that a regular $\FF_p$-algebra is a 
derived splinter by \cite[Thm. 2]{Hoc73}. Regular algebras essentially of finite
type over a field of characteristic $0$ are derived splinters by \cite{Kov00} (see
also \cite[Thm. 2.12]{Bha12}) because rational singularities are derived 
splinters. To show that an arbitrary regular $\mathbb{Q}$-algebra $R$ is a derived 
splinter, we may assume that $R$ is local by Lemma \ref{lem:d-splinter-local} and 
further that $R$ is complete by faithfully flat descent of the derived splinter 
property (Proposition \ref{prop:descent-derived-splinter}). That a complete regular 
local $\mathbb{Q}$-algebra is a derived splinter now follows from 
\cite[Thm. 5.11]{Ma18}, in which Ma relates the derived splinter property to a 
vanishing condition on maps of Tor. 
\end{proof}

\begin{proof}[Third proof of Theorem \ref{thm:valuations-derived-splinters}]
Let $V$ be a valuation ring and $f: X \rightarrow \Spec(V)$ be a proper surjective
morphism.Via Remark \ref{rem:flat-fp} we may further assume that $f$ is proper, surjective,
flat and finitely presented. 
Then, as in the proof of Theorem \ref{thm:valuations-derived-splinters} that uses
the valuative criterion, we may assume that $V$ is absolutely integrally closed.
 Write $V$ as a colimit of regular rings 
$\{R_i\}_{i \in I}$ by Corollary~\ref{cor:absoluteintegralclosed}, where $I$ is a directed set. 
Since $f$ is finitely presented, by \cite[\href{https://stacks.math.columbia.edu/tag/01ZM}{Tag 01ZM}]{stacks-project}
there exists some $i_0 \in I$ and a finitely presented morphism 
$f_{i_0}: X_{R_{i_0}} \rightarrow \Spec(R_{i_0})$ such that
$$X = X_{R_{i_0}} \times_{R_{i_0}} V.$$  
Then by \cite[\href{https://stacks.math.columbia.edu/tag/04AI}{Tags 04AI}, \href{https://stacks.math.columbia.edu/tag/07RR}{07RR}, \href{https://stacks.math.columbia.edu/tag/081F}{081F}]{stacks-project}, one can find $i_1 \geq i_0$ such that the base change of $f_{i_0}$ along
$\Spec(R_{i_1}) \rightarrow \Spec(R_{i_0})$ is flat, surjective and proper. 
The upshot is that there exists some $i \in I$ and a cartesian square
$$\xymatrix{
        X\ar[r]^{g'}\ar[d]^{f}&X_{R_i}\ar[d]^{f_{R_i}}\\
        \Spec(V)\ar[r]^{g}&\Spec(R_i)
    }$$
where the right vertical map $f_{R_i}$ is also proper, surjective, flat and finitely 
presented. Since $R_i$ is a derived splinter by Theorem \ref{thm:derived-direct-summand},
the map $\cO_{\Spec(R_i)} \rightarrow \bR (f_{R_i})_* \cO_{X_{R_i}}$ splits. Applying
$\bL g^*$ to this splitting then gives us a splitting of 
$\cO_{\Spec(V)} \rightarrow \bR f_* \cO_X$ using Theorem \ref{thm:derived-pushforward-fpqc-pullback}(b)
(which applies because of flatness of $f_{R_i}$).
\end{proof}

A similar proof technique leads to the following result:

\begin{theorem}
\label{thm:smooth-valuations}
Let $V \rightarrow A$ be a smooth map of commutative rings where $V$ is a valuation ring. 
\begin{enumerate}
    \item[{\em (i)}] The ring $A$ is a splinter.
    \item[{\em (ii)}] If $f: X \rightarrow \Spec(A)$ is a proper finitely presented surjective morphism
	that is flat, then $\cO_{\Spec(A)} \rightarrow 
	\bR f_*\cO_X$ splits in $\Dscr_{\qc}(\Spec(A))$.
\end{enumerate}
\end{theorem}

\begin{proof}
Choose an absolutely integrally closed valuation ring $V'$ with a faithfully flat map $V\rightarrow V'$.
By pure descent of the splinter
property (Proposition \ref{prop:decent-colimit-splinter}(i)), it suffices 
to show that $V' \otimes_V A$ (which is a smooth $V'$ algebra) is a splinter.
Thus, we may assume $V$ is absolutely integrally closed. By Proposition 
\ref{prop:decent-colimit-splinter}(ii) and the direct summand theorem,
it suffices to show that $A$ is a directed colimit of regular rings. Write
$V = \colim_I R_i$ as a colimit of regular rings $R_i$  over a directed
set $I$ (using Corollary~\ref{cor:absoluteintegralclosed}). Since smooth
maps are finitely presented, choose a finitely presented algebra $A_i$ over
some $R_i$ such that $A = A_i \otimes_{R_i} V$. Then by 
\cite[\href{https://stacks.math.columbia.edu/tag/0C0C}{Tag 0C0C}]{stacks-project}
one can assume without loss of generality that $A_i$ is smooth over $R_i$.
Clearly $$A = \colim_{j \geq i} A_i \otimes_{R_i} R_j,$$ where
$A_i \otimes_{R_i} R_j$ is a smooth $R_j$-algebra for each $j \geq i$ because 
smoothness is preserved under base change 
\cite[\href{https://stacks.math.columbia.edu/tag/00T4}{Tag 00T4}]{stacks-project}. 
Now the fact that a smooth algebra over a regular ring is regular 
\cite[\href{https://stacks.math.columbia.edu/tag/07NF}{Tag 07NF}]{stacks-project} 
    completes the proof of (i).

Note that a smooth algebra $A$ over a valuation ring is universally cohesive.
Thus, by the proof of faithfully flat descent of the derived splinter property
for a universally cohesive base (Proposition \ref{prop:descent-derived-splinter}), 
it suffices to assume $V$ is absolutely 
integrally closed as in (i).
Now one proves (ii) by using the ind-regularity of $A$ from the proof of (i) 
and then repeating the argument 
in the third proof of Theorem \ref{thm:valuations-derived-splinters} above.
\end{proof}

\begin{remark}
{\*}
\begin{enumerate}
    \item[(1)] A smooth $V$-scheme need not be a splinter (hence also not a derived splinter). 
For instance, any smooth projective variety $X$ over an algebraically closed 
field of prime characteristic $p > 0$ with ample $\omega_X$ cannot be a splinter
\cite[Cor. 2.18]{SZ15}. This 
is because the splinter condition implies that $X$ is globally Frobenius
split (the absolute Frobenius is a finite surjective morphism), which by a standard 
Grothendieck duality argument then shows that $\omega_X^{\otimes (1-p)}$ has 
non-trivial sections.

\item[(2)] One can adapt the arguments of this subsection to show that if
a ring $A$ is a directed colimit of a collection of rings $\{A_i\}_{i \in I}$ where 
each $A_i$ is a derived splinter, then $A$ partially satisfies the condition of
being a derived splinter in the sense that the splitting condition holds for any proper, 
finitely presented, surjective morphism $X \rightarrow \Spec(A)$ that is also flat. If the
transition maps of the directed system $\{A_i\}_{i \in I}$ are in addition flat, 
then one can conclude that $A$ is indeed a derived splinter. For example, the perfection
of a regular ring of prime characteristic is a derived splinter.

This raises the following 
natural question: \emph{is a directed colimit of derived splinter rings also a derived
splinter?}

The issue here is being able to 
pick a model of a proper, finitely presented $A$-scheme $X$ over some $A_i$
such that the associated pullback square is Tor-independent. 
\end{enumerate}
\end{remark}

\small
\bibliographystyle{amsplain}
\bibliography{dsplint}

\vspace{4mm}

\noindent Benjamin Antieau\\
\texttt{benjamin.antieau@gmail.com}\\

\noindent Rankeya Datta\\
\texttt{rankeya@uic.edu}

\end{document}